\documentclass{article}
\usepackage{graphicx} 
\usepackage{amsfonts}
\usepackage{amssymb}
\usepackage{amsmath}
\usepackage{amsthm}
\usepackage[margin=1in]{geometry}
\usepackage{indentfirst}
\usepackage{mathtools}
\mathtoolsset{showonlyrefs}
\usepackage{authblk}
\usepackage{xcolor}

\usepackage{hyperref}
\hypersetup{citecolor=blue, urlcolor=cyan, colorlinks=true, linkcolor=blue}

\newcommand{\beu}{\begin{equation*}}
\newcommand{\eeu}{\end{equation*}}
\newcommand{\be}{\begin{equation}}
\newcommand{\ee}{\end{equation}}

\theoremstyle{plain}
\newtheorem{theorem}{Theorem}[section]
\newtheorem{lemma}[theorem]{Lemma}
\newtheorem{proposition}[theorem]{Proposition}
\newtheorem{corollary}[theorem]{Corollary}
\newtheorem{assumption}[theorem]{Assumption}
\newtheorem{example}[theorem]{Example}
\newtheorem{definition}[theorem]{Definition}

\theoremstyle{remark}
\newtheorem{remark}[theorem]{Remark}

\title{Exponential Convergence of the Sinkhorn Algorithm for the Schrödinger Bridge with Regime Switching}
\author[1]{Katharina Eichinger}
\author[2]{Anna Kazeykina}
\author[3]{Zhenjie Ren}
\author[2]{Hecheng Wang}
\affil[1]{ParMA, Inria and Université Paris Saclay}
\affil[2]{LMO, Université Paris-Saclay}
\affil[3]{LaMME, Université \'Evry Paris-Saclay}

\date{}

\begin{document}

\maketitle

\begin{abstract}
This paper studies the convergence of the Sinkhorn algorithm for the Schr\"odinger bridge problem with regime switching, as introduced in \cite{zlotchevski2025schrodingerbridgeproblemjump}. We consider a class of regime-switching stochastic systems on the hybrid state space \(E=\mathbb{R}^d\times\{1,\ldots,m\}\), and construct the Sinkhorn iteration through the associated equivalent entropic optimal transport formulation. 
The main result of this paper is the exponential convergence of the Sinkhorn algorithm in relative entropy under compactness assumptions. 
Our proofs are inspired by the arguments recently developed for proving exponential convergence of the classical Schrödinger problem in \cite{ChiariniConfortiGrecoTamanini2024SemiconcStabConv, eckstein2025hilbert}.
We perform a similar analysis for the partially observed terminal setting, where only the marginal distribution of the continuous component is prescribed at the terminal time, while the discrete regime is unobserved. 

\end{abstract}



\section{Introduction}

The Schrödinger bridge problem, introduced by Schrödinger in 1931 \cite{Schrodinger1931,Schrodinger1932}, asks for the most likely stochastic evolution connecting two prescribed endpoint distributions. In modern terms, it seeks, among all path measures satisfying the endpoint marginal constraints, the one closest in relative entropy to a reference diffusion process \cite{Follmer1988}. It has become popular in recent year due to its proximity to optimal transport \cite{leonard2014survey} and efficiency to compute solutions \cite{cuturi2013sinkhorn}.
This framework also provides a flexible way to construct stochastic dynamics between probability distributions and has recently attracted substantial interest in machine learning, especially in generative modeling, where it connects diffusion-type models, stochastic control, and distribution matching \cite{Wang2021DeepSB,DeBortoli2021DiffusionSB,Chen2022SBFBSDE}. Beyond noise-to-data generation, Schrödinger bridges have been applied to image translation and restoration, as well as to other scientific problems where intermediate dynamics are inferred from marginal observations, such as single-cell trajectory inference, cell differentiation, and protein conformational changes \cite{Liu2023I2SB,Somnath2023AlignedDSB,Tong2024SimulationFreeSB}. 

However, many existing Schrödinger bridge models are formulated in the classical Euclidean setting, where each sample is described only by a continuous position and follows a common stochastic evolution mechanism. This can be restrictive for heterogeneous systems. A natural extension is to attach a regime label to each sample, representing its class, condition, environment, or latent mode. The label may remain fixed or evolve over time, leading to a multi-regime formulation in which the continuous dynamics depend on the current regime and the regime itself may switch.

In this paper, we consider a multi-regime version of the Schrödinger bridge
problem, as introduced in \cite{zlotchevski2025schrodingerbridgeproblemjump}. In contrast to the classical setting where the state variable takes
values only in Euclidean space, we consider the hybrid state space $E=\mathbb{R}^d\times \mathcal I$,
where \(\mathcal I=\{1,\ldots,m\}\) is a finite set of regimes. Thus, each
sample carries not only a continuous position \(x\in\mathbb{R}^d\), but also a
discrete state \(i\in\mathcal I\). Accordingly, the Euclidean diffusion is
replaced by a regime-switching diffusion \cite{YinZhu2010}. Intuitively, the particle
diffuses continuously within its current regime, while a Poisson jump mechanism
switches the particle between different regimes; the jump changes only the
regime index \(I_t\), without directly changing the continuous position \(X_t\).
More precisely, we consider a hybrid-state process \(Z_t=(X_t,I_t)\) whose
dynamics are given by
\[
\left\{
\begin{aligned}
& dX_t
=
b_{I_t}(t,X_t)\,dt+\sigma_{I_t}(t,X_t)\,dW_t,\\
& \mathbb{P}(I_{t+h} = j \mid I_t = i, X_t = x) = \lambda_{ij}(t, x) h + o(h), \quad h \downarrow 0,
\end{aligned}
\right.
\]
where \(W_t\) is a Brownian motion and \(\lambda_{ij}(t,x)\) is the switching intensity
from regime \(i\) to regime \(j\) if the current state is $ ( X_t, I_t ) = (x, i) $. Therefore,
the reference process contains both the continuous diffusion mechanism within
each regime and the random switching mechanism between regimes.

As in the classical Schrödinger bridge problem, our goal is to find the most
likely path distribution under prescribed endpoint observations. Let \(P\) be
the path law of the reference regime-switching diffusion. Given two endpoint
marginal distributions
$
\rho,\mu\in\mathcal P(\mathbb{R}^d\times\mathcal I),
$
the fully observed multi-regime Schrödinger bridge problem is formulated as
\begin{equation}
\label{eq:schrodinger_bridge}
\inf_{Q\ll P} \left\{ \mathcal H(Q\mid P): Q\circ Z_0^{-1}=\rho,\ Q\circ Z_T^{-1}=\mu \right\},
\end{equation}
where \(\mathcal H(Q\mid P)\) denotes the relative entropy of the path measure \(Q\) with
respect to the reference path measure \(P\). This formulation means that, among
all stochastic evolutions matching the prescribed initial and terminal
distributions, we select the one with minimal deviation in relative entropy from the reference dynamics.

Diffusion models with regime switching have multiple applications in epidemic modeling \cite{TuongNguyenDieuTran2019SIRSRegimeSwitching}, population dynamics \cite{LuoMao2007StochasticPopulationRegimeSwitching}, financial modeling \cite{Zhang2001StockTradingOptimalSelling}, risk management \cite{HuPang2025RiskManagementGMB}; the Schrödinger bridge problem for such models also naturally arises in the problem of allocation of drones (see e.g. \cite{kazeykina2025eotconvex} for the context). 

One of natural applications of the mathematical problem considered in the present paper is longitudinal medical imaging and disease-progression modeling \cite{tay2025shape,oxtoby2023disease}. In this setting, the continuous variable $X_t$ may represent a medical image, or a high-dimensional feature extracted from medical images, while the discrete variable $I_t$ represents a latent clinical state, such as disease stage, treatment status, or recovery state. Thus, an endpoint $(x,i)$ encodes both the observed imaging state and the underlying clinical regime, and a multi-regime Schrödinger bridge can model stochastic evolution between patient populations whose imaging distributions and clinical states change over time. Compared with a single-regime model, regime switching naturally represents disease progression, treatment response, or recovery through random transitions between clinical states, in the spirit of multistate and continuous-time hidden Markov models for disease dynamics \cite{liu2015cthmm,therneau2024multistate}. 

In longitudinal studies the clinical regime may be observed only intermittently, while lower-cost measurements such as blood pressure, heart rate, respiratory indices, or other physiological signals may be available more frequently; this motivates a setting in which the initial hybrid state is observed, but at the terminal time only the continuous component is prescribed \cite{hhs2024rpm}.
This naturally leads to the partially observed multi-regime Schrödinger bridge
considered in this paper, which has not been studied in the literature up to our knowledge. Given an initial distribution
$\rho\in\mathcal P(\mathbb{R}^d\times\mathcal I)$
and a terminal marginal distribution
$\mu_{\mathrm p}\in\mathcal P(\mathbb{R}^d)$
for the continuous component, the partially observed problem is defined by
\begin{equation}
\label{eq:partial_schrodinger_bridge}
\inf_{Q\ll P} \left\{ \mathcal H(Q\mid P): Q\circ Z_0^{-1}=\rho, \quad Q\circ X_T^{-1}=\mu_{\mathrm p} \right\}.
\end{equation}

Numerically, the equivalence between Schrödinger bridges and entropic
optimal transport \cite{cuturi2013sinkhorn} provides a natural way to solve Schrödinger bridge
problems through the Sinkhorn algorithm. More precisely, after reducing the
dynamic problem \eqref{eq:schrodinger_bridge} to the endpoint space and setting
$c_{ij}(x,y):=-\log r_{ij}(0,x;T,y)$, where $ r_{ij}(0,x;T,y) $ is the transition density of the reference process,
the static Schrödinger problem is equivalent, up to an additive constant, to
\begin{equation}\label{eq:main}
\inf_{\pi\in\Pi(\rho,\mu)}
\left\{
\int_{(\mathbb{R}^d\times\mathcal{I})^2}
c(z_0,z_T)\,\pi(dz_0,dz_T)
+
H(\pi\mid \rho\otimes\mu)
\right\},
\end{equation}
where $\Pi(\rho,\mu)$ denotes the set of probability measures on $ E \times E $ with prescribed endpoint marginals $ \rho $ and $ \mu $.
This EOT formulation is the basis for the Sinkhorn iteration studied in this
paper.

There is a vast literature on the analysis of Sinkhorn’s algorithm for entropic optimal transport. In this work, we restrict attention to the results most closely related to our setting. A recent contribution in this direction is due to \cite{GhosalNutz2025SinkhornRate}, that establishes non-asymptotic convergence rates for Sinkhorn’s algorithm in relative entropy for a broad class of state spaces and cost functions. However, the assumptions imposed there may not be straightforward to verify in some concrete applications. We also mention the earlier work \cite{NutzWiesel2023StabilitySchrodingerPotentials}, which proves convergence of Sinkhorn iterates in total variation, though without an explicit convergence rate. These results are closely related to the stability theory developed in \cite{EcksteinNutz2022QuantitativeStability}, that proves quantitative stability estimates for entropic optimal transport and applies them to obtain the convergence of Sinkhorn’s algorithm in Wasserstein distance.

The present work is principally inspired by \cite{ChiariniConfortiGrecoTamanini2024SemiconcStabConv}, where semiconcavity properties of the cost function and of the Sinkhorn potentials are used to prove stability of entropic plans and exponential convergence of Sinkhorn’s algorithm. This work builds in turn on \cite{conforti2024weak}; see also \cite{ConfortiDurmusGreco2023SinkhornContraction, GrecoNobleConfortiDurmus2023SinkhornGradients, ChaintronConfortiEichinger2025} for closely related works on the convergence of Sinkhorn algorithm and propagation of semiconcavity. In the same context, \cite{GrecoTamanini2025Hessian} is also worth mentioning, as it provides a dynamic proof of quantitative stability results, only requiring one-side Hessian bounds for the Schro\"dinger potentials. Compared with the aforementioned results, our model requires the analysis of semiconcavity properties for solutions of a larger coupled system of equations. Whether such properties hold in this setting is not clear and constitutes one of the main additional difficulties.

The present paper draws, among other sources, on the work \cite{eckstein2025hilbert}, where exponential convergence of Sinkhorn’s algorithm in total variation distance is established using Hilbert projective metric methods. We use these results as one of the building blocks for deriving stronger estimates in terms of relative entropy.

The interest in Schro\"dinger bridge problem with different regimes is relatively recent. 
Several related works include \cite{eldesoukey2024excursion,EldesoukeyMiangolarraGeorgiou2025}, which study Schrödinger bridge problems with a killing mechanism. In fact, this model can be viewed as a degenerate form of a two-regime Schrödinger bridge: after particles enter one of the regimes, their evolution is completely stopped, in the sense that both the drift and the diffusion coefficient vanish. The authors derive a Sinkhorn-type algorithm which is not equivalent to the standard EOT-based Sinkhorn algorithm, and prove its convergence using Hilbert's projective metric.

The multi-regime Schrödinger problem was first studied in the recent paper  \cite{zlotchevski2025schrodingerbridgeproblemjump}, which builds upon \cite{zlotchevski2024schrodingerjump}. The paper focuses on the dynamic structure of the optimal bridge: the authors derive the related SDE, generator, stochastic control formulation, and the corresponding dynamic Schrödinger system (see also \cite{zlotchevski2025switchingunbalanced} for a further work on the regime-switching approach to the unbalanced Schr\"odinger bridge problem). In contrast, the present paper's contribution is more algorithmic and quantitative:
we study the Sinkhorn algorithm associated with the regime-switching
Schrödinger bridge problem and prove its exponential convergence in
relative entropy.

The rest of the paper is organised as follows. In Section \ref{sec:RS} we give a detailed description of the studied model, introduce the necessary notations and assumptions, and present the main results. Section \ref{sec:proof-RS} is devoted to the proofs of the two main theorems on stability of the regime-switching Schrödinger bridge problem and exponential convergence of the associated Sinkhorn algorithm. Section \ref{sec:proof-partial} contains the proofs of results on partial terminal observations. Finally, in the Appendix we give a proof of the regularity of the transition density of the reference process in a restricted constant-coefficient case. 

\paragraph{Acknowledgments}
KE, AK and ZR acknowledge the support of the Agence nationale de la recherche, through the PEPR PDE-AI project (ANR-23-PEIA-0004). AK's research is supported by the FMJH Program Gaspard Monge for optimization and operations research and their interactions with data science.  ZR's research is supported by the Finance For Energy Market Research Center and 
the France 2030 grant (ANR-21-EXES-0003).

\section{Hybrid-State Diffusion Processes and Schrödinger Bridge}
\label{sec:RS}

\subsection{Preliminaries}
 
  Our reference dynamics is given by a \textbf{hybrid-state process} $Z_t = (X_t, I_t)$ taking values in the product space $E := \mathbb{R}^d \times \mathcal{I}$, where $\mathcal{I} = \{1, \dots, m\}$ is a finite set of regimes. This process consists of a continuous component $X_t \in \mathbb{R}^d$ and a discrete regime component $I_t \in \mathcal{I}$. The dynamics are governed by the following coupled system governed by drift $ b_i:[0,T]\times\mathbb R^d\to\mathbb R^d $, diffusion coefficient $ \sigma_i:[0,T]\times\mathbb R^d\to\mathbb R^{d\times d} $
and jump rate
$ \lambda_{ij}:[0,T]\times\mathbb R^d\to\mathbb R_+$, for $i, j\in \mathcal{I}$:

\begin{enumerate}
    \item \textbf{Continuous Diffusion:} For a given regime $I_t = i$, the state $X_t$ evolves according to the SDE:
    \begin{equation}
    \label{eq:diffusion_eq}
        dX_t = b_i(t, X_t) dt + \sigma_i(t, X_t) dW_t,
    \end{equation}
    where $W_t$ is a $d$-dimensional Brownian motion.
    
    \item \textbf{Regime Switching:} The pair $(X_t, I_t)$ is a Markov process on $\mathbb{R}^d \times \mathcal{I}$. Its discrete component $I_t$ switches according to state-dependent intensities:
    \begin{equation*}
        \mathbb{P}(I_{t+h} = j \mid I_t = i, X_t = x) = \lambda_{ij}(t, x) h + o(h), \quad h \downarrow 0.
    \end{equation*}
\end{enumerate}

This heuristic description can be made rigorous through the Poisson random
measure construction. Following \cite{zlotchevski2025schrodingerbridgeproblemjump},  let \(N_1(dt,du)\) be a Poisson random measure on \([0,T]\times\mathbb{R}_+\) with intensity \(dt\,du\). Its atoms are pairs \((t,u)\), where \(t\) is the time of a potential switching event and \(u\) is an auxiliary mark used only to determine the type of the switch.

For every \((t,x,i)\), choose pairwise disjoint Borel sets \(\Delta_{ij}(t,x)\subset\mathbb{R}_+\), \(j\neq i\), such that
\begin{equation}
\lvert \Delta_{ij}(t,x)\rvert=\lambda_{ij}(t,x),
\end{equation}
and set \(\Delta_{ii}(t,x)=\varnothing\). Define
\begin{equation}
\alpha(t,x,i,u)
=
\sum_{j\in\mathcal I}
(j-i)\mathbf{1}_{\Delta_{ij}(t,x)}(u).
\end{equation}
Thus, \(\alpha(t,x,i,u)\) is the increment of the regime index caused by the mark \(u\). More precisely, if the current regime is \(i\) and \(u\in\Delta_{ij}(t,x)\), then \(\alpha(t,x,i,u)=j-i\), so that the regime changes from \(i\) to \(j\). If \(u\) belongs to none of these sets, then \(\alpha(t,x,i,u)=0\), and no switch occurs.

The regime process is defined by
\begin{equation}
\label{eq:jump_eq}
I_t
=
I_0+
\int_0^t\int_{\mathbb{R}_+}
\alpha(s,X_{s-},I_{s-},u)\,N_1(ds,du).
\end{equation}

The sample paths of the joint process $Z = (Z_t)_{t \in [0,T]}$ lie in the \textbf{canonical path space} $\Omega_T$, defined as the product of the space of continuous functions and the space of càdlàg (right-continuous with left limits) functions:
\begin{equation*}
    \Omega_T := D([0, T]; \mathbb{R}^d \times \mathcal{I}) = C([0, T]; \mathbb{R}^d) \times D([0, T]; \mathcal{I}).
\end{equation*}
We denote by $P$ the \textbf{path law} of the process $Z$, which is a probability distribution on $(\Omega_T, \mathcal{F}_T)$, where $ \mathcal{F}_T $ is the canonical filtration on $ \Omega_T $.

The classical Schrödinger Bridge Problem (SBP) aims to identify a path law that minimally deviates from the reference process $P$ while interpolating between two given endpoint distributions. In the present hybrid-state setting with regime switching, given
\(\rho,\mu\in\mathcal P(\mathbb R^d\times\mathcal I)\), the
\textbf{regime-switching Schrödinger bridge problem} (RS-SBP) is given by
\begin{equation}
\inf_{Q\ll P} \left\{ \mathcal H(Q\mid P): Q\circ Z_0^{-1}=\rho,\ Q\circ Z_T^{-1}=\mu \right\},
\end{equation}

\noindent Here, $ \mathcal{H} $ denotes the relative entropy between probability measures on the path space $\Omega_T$, more precisely,
\[
\mathcal H(Q\mid P)
:=
\begin{cases}
\displaystyle
\int_{\Omega_T}
\log\!\left(\frac{\mathrm{d}Q}{\mathrm{d}P}\right)
\,\mathrm{d}Q,
& Q\ll P,\\[1ex]
+\infty,
& \text{otherwise}.
\end{cases}
\]

With the usual argument \cite{leonard2014survey}, the dynamic RS-SBP is equivalent to a static entropy optimal transport
problem on the endpoint space. For any admissible path measure $Q \in \mathcal{P}(\Omega_T)$ with $\mathcal H(Q\mid P) < + \infty$, let
\[
\pi := Q\circ (Z_0,Z_T)^{-1},
\qquad
P_{0,T}:=P\circ (Z_0,Z_T)^{-1},
\]
where $Z_t:=(X_t,I_t)$. The additive property of the relative entropy yields
\begin{equation}\label{eq:entropy_decomposition}
\mathcal H(Q\mid P) =
H(\pi\mid P_{0,T}) + \int \mathcal H\bigl(Q(\cdot\mid Z_0 = z_0, Z_T = z_T)
\mid
P(\cdot\mid Z_0 = z_0, Z_T = z_T)
\bigr)
\,\pi(dz_0,dz_T),
\end{equation}
where $H$ stands for the
relative entropy between probability measures on the endpoint space $(\mathbb R^d\times \mathcal I)^2$.
Hence,
$$
H(\pi\mid P_{0,T}) \leq \mathcal H(Q\mid P),
$$
with equality if and only if $Q(\cdot\mid Z_0 = z_0, Z_T = z_T) = P(\cdot\mid Z_0 = z_0, Z_T = z_T)$ for $\pi$-a.e. $(z_0,z_T)$.
Conversely, if $\pi \in \Pi(\rho,\mu)$ is a coupling, then
\begin{equation}\label{eq:static-to-dyn}
Q(\cdot) = \int_{(\mathbb R^d\times \mathcal I)^2} P(\cdot\mid Z_0 = z_0, Z_T = z_T) \pi(dz_0,dz_T)
\end{equation}
provides a corresponding law on path space with
$H(\pi\mid P_{0,T}) = \mathcal H(Q\mid P).$
In particular, if there is an optimal bridge $Q^{\rho,\mu}$ on path space
\begin{equation}
\label{eq:optimal_bridge_conditionals}
Q^{\rho,\mu}(\cdot\mid Z_0 = z_0, Z_T = z_T)
=
P(\cdot\mid Z_0 = z_0, Z_T = z_T),
\end{equation}
for $\pi^{\rho,\mu}$-a.e. $(z_0,z_T)$, where $\pi^{\rho,\mu} := Q^{\rho,\mu}\circ (Z_0,Z_T)^{-1}$. It follows that
\[
\mathcal H(Q^{\rho,\mu}\mid P)
=
H(\pi^{\rho,\mu}\mid P_{0,T}),
\]
and the dynamic RS-SBP is equivalent to the static problem
\begin{equation}
\label{eq:static_SSBP}
\inf_{\pi\in\Pi(\rho,\mu)}
H(\pi\mid P_{0,T}), 
\end{equation}
where $ \Pi( \rho, \mu ) $ denotes the set of probability measures on $ E \times E $ with prescribed endpoint marginals $ \rho $ and $ \mu $, i.e. 
$$
\Pi( \rho, \mu ) =\{ \pi \in \mathcal{P}( E \times E ) \colon \, \pi( A \times E ) = \rho( A ), \, \pi( E \times B ) = \mu( B ) \}
$$
for all measurable $ A, B \subset E $.

It is hence sufficient to work with the static formulation since the properties can be readily translated back to the path space formulation via \eqref{eq:static-to-dyn}.

We now recall a result of \cite[Theorem 2.1]{nutz2022eot} providing the classical characterization of the static regime-switching Schrödinger bridge.

\begin{assumption}
\label{assum:reference_marginals}
Let $
P_{0,T}:=P\circ (Z_0,Z_T)^{-1} $
be the endpoint law of the reference process. Assume that
\[
H(\rho\otimes\mu\mid P_{0,T})<\infty \quad \text{ and } \quad
P_{0,T}\big|_{\operatorname{supp}(\rho\otimes\mu)}
\ll \rho\otimes\mu.
\]
\end{assumption}

Define
\[
p(z_0,z_T)
:=
\frac{
d\!\left(
P_{0,T}\big|_{\operatorname{supp}(\rho\otimes\mu)}
\right)
}{
d(\rho\otimes\mu)
}(z_0,z_T).
\]
In componentwise notation, for $z_0=(x,i)$ and $z_T=(y,j)$, write
\[
p_{ij}(x,y):=p\big((x,i),(y,j)\big).
\]

\begin{theorem}[Characterization of the static regime-switching Schr\"odinger bridge,
\cite{nutz2022eot}]
\label{thm:classic}
Under Assumption \ref{assum:reference_marginals}
the static regime-switching Schr\"odinger bridge problem \eqref{eq:static_SSBP} admits a unique optimal
solution $\pi^{\rho,\mu}\ll P_{0,T}$, characterized as follows.
The Radon--Nikodym derivative of the optimal measure $\pi^{\rho,\mu}$ with respect to the reference measure $P_{0,T}$ satisfies
\begin{equation}
   \frac{d\pi^{\rho,\mu}}{dP_{0,T}}(z_0,z_T) = \tilde{f}_{i_0}(x_0) \tilde{g}_{i_T}(x_T) \quad \text{for }\pi^{\rho,\mu}\text{-a.e. } (z_0,z_T)=((x_0,i_0),(x_T,i_T)).
\end{equation}
where $(\tilde{f}_i)_{i \in \mathcal I}$ and $(\tilde{g}_j)_{j \in \mathcal I}$ are nonnegative measurable functions, unique up to the transformation
\begin{equation}
    (\tilde{f}, \tilde{g}) \sim (\lambda \tilde{f}, \lambda^{-1} \tilde{g}), \quad \lambda > 0,
\end{equation}
and solve the Schr\"odinger system
\begin{equation}
\begin{cases}
    1 = \tilde{f}_i(x) \displaystyle\sum_{j \in \mathcal I} \int_{\mathbb{R}^d} p_{ij}(x, y) \tilde{g}_j(y) \mu_j(dy), & \rho_i\text{-a.e. } x, \\
    1 = \tilde{g}_j(y) \displaystyle\sum_{i \in \mathcal I} \int_{\mathbb{R}^d} p_{ij}(x, y) \tilde{f}_i(x) \rho_i(dx), & \mu_j\text{-a.e. } y.
\end{cases}
\end{equation}
\end{theorem}

In particular, the optimal path measure $Q^{\rho,\mu}$ uniquely exists as well and satisfies thanks to the discussion above
\begin{equation*}
    \frac{dQ^{\rho,\mu}}{dP}((Z_t)_{t\in[0,T]})=\frac{d\pi^{\rho,\mu}}{dP_{0,T}}(Z_0,Z_T) \quad Q^{\rho,\mu}\text{-a.e. } (Z_t)_{t\in[0,T]}
\end{equation*}

\begin{remark}
Note that \cite{zlotchevski2025schrodingerbridgeproblemjump} gives a result on the characterization of the optimal path measure $Q^{\rho,\mu}$ as a system of SDEs min terms of optimal drift $b^\ast$, volatility $\sigma^\ast$ and switching rates $\lambda^\ast$ of the associated process. 

Let the function \(h=(h_i)_{i\in\mathcal I}\) be defined by
\[
h_i(t,x)
:=
\mathbb{E}_{\mathbb{P}}
\!\left[
\tilde g_{I_T}(X_T)
\mid
X_t=x,\ I_t=i
\right],
\qquad
(t,x,i)\in[0,T]\times\mathbb{R}^d\times\mathcal I.
\]
Under some mild conditions on the drift $b$, volatility $\sigma$ and switching rates $\lambda$ of the reference process, the strong Markov property for the references process and regularity, positivity and harmonicity assumptions on $h$ (see Assumptions (C), (G) and (H) of \cite{zlotchevski2025schrodingerbridgeproblemjump}), we have that the optimal triplet
\(
(b^\ast,\sigma^\ast,\lambda^\ast)
\)
inducing the path law \(Q^{\rho,\mu}\) is given by:
    \begin{equation}
    \begin{aligned} 
    \sigma_i^*(t, x) &= \sigma_i(t, x), \\ 
    b_i^*(t, x) &= b_i(t, x) + a_i(t, x) \nabla_x \log h_i(t, x), \\ 
    \lambda_{ij}^*(t, x) &= \lambda_{ij}(t, x) \frac{h_j(t, x)}{h_i(t, x)}, \quad j \neq i, 
    \end{aligned}
    \end{equation}
where
\[
a_i(t,x)
=
\sigma_i(t,x)\sigma_i(t,x)^\top
\]
and the function \(h\) satisfies the following
coupled Kolmogorov backward system:
\begin{equation}
    \begin{cases} 
    \partial_t h_i(t, x) + b_i(t, x) \cdot \nabla_x h_i(t, x) + \frac{1}{2} \text{Tr}(a_i(t, x) \nabla_x^2 h_i(t, x)) \\ 
    \qquad + \displaystyle \sum_{j \neq i} \lambda_{ij}(t, x) \left( h_j(t, x) - h_i(t, x) \right) = 0, & (t, x, i) \in [0, T) \times \mathbb{R}^d \times \mathcal{I}, \\ 
    h_i(T, x) = \tilde g_i(x), & x \in \mathbb{R}^d, i \in \mathcal{I}.
    \end{cases}
    \end{equation}

\end{remark}

Let us now draw the link to entropic optimal transport, and identify the corresponding cost function.
Let \(\rho_{\mathrm{ref}}:=P\circ Z_0^{-1}\) be the initial marginal of the reference process, and denote by \(\rho_{\mathrm{ref},i}\) the density of its \(i\)-th component with respect to the Lebesgue measure. From now on, we assume that the transition kernel of the reference process admits a density with respect to the Lebesgue measure
\begin{equation}
\label{eq:trans_density}
P(X_T \in dy, I_T = j \mid X_0 = x, I_0 = i) = r_{ij}(0,x;T,y)dy.
\end{equation}
Later on, we impose stronger conditions on the reference process to guarantee our main theorems (see Assumption \ref{assum:map_c2}), so we do not go into technical assumptions when this holds true to avoid technical overload.

The reference endpoint density $p_{ij}(x,y)$, now understood as the density of $P_{0,T}\big|_{\operatorname{supp}(\rho\otimes\mu)}$ with respect to
$\rho\otimes\mu$, can be decomposed as
\[
p_{ij}(x,y)
=
\frac{\rho_{\mathrm{ref},i}(x) r_{ij}(0,x;T,y)}
{\rho_i(x)\mu_j(y)}.
\]

For any \(\pi\in\Pi(\rho,\mu)\) such that \(H(\pi\mid P_{0,T})<\infty\), writing $\pi_{ij}(x,y)$ for the density of $\pi_{ij}$ with respect to
$\rho_i\otimes\mu_j$, we have
\[
\begin{aligned}
H(\pi \mid P_{0,T}) 
&= \sum_{i,j \in \mathcal{I}} \iint \log\!\left( \frac{\pi_{ij}(x,y)}{\frac{\rho_{\mathrm{ref},i}(x) r_{ij}(0,x;T,y)}{\rho_i(x)\mu_j(y)}} \right) \pi_{ij}(x,y) \, \rho_i(\mathrm{d}x)\mu_j(\mathrm{d}y) \\
&= \sum_{i,j \in \mathcal{I}} \iint \log\!\left( \frac{\pi_{ij}(x,y)}{\rho_i(x)\mu_j(y)} \right) \pi_{ij}(x,y) \, \rho_i(\mathrm{d}x)\mu_j(\mathrm{d}y) \\
&\quad - \sum_{i,j \in \mathcal{I}} \iint \log r_{ij}(0,x;T,y) \, \pi_{ij}(x,y) \, \rho_i(\mathrm{d}x)\mu_j(\mathrm{d}y) + \mathrm{constant},
\end{aligned}
\]
where for suitable functions $f_i,g_j: \mathbb{R}^d \rightarrow \mathbb{R}$
\[
\begin{aligned}
\mathrm{constant}
&=
\sum_{i,j\in\mathcal I}
\iint
\bigl(f_i(x)+g_j(y)\bigr)\pi_{ij}(x,y)\,
\rho_i(\mathrm dx)\mu_j(\mathrm dy) \\
&=
\sum_{i\in\mathcal I}\int f_i(x)\,\rho_i(\mathrm dx)
+
\sum_{j\in\mathcal I}\int g_j(y)\,\mu_j(\mathrm dy).
\end{aligned}
\]
Hence the constant is independent of \(\pi\), since
\(\pi\in\Pi(\rho,\mu)\).

Define the cost function by
\begin{equation}\label{eq:def-cost}
c_{ij}(x,y):=-\log r_{ij}(0,x;T,y).
\end{equation}
Then
\[
H(\pi\mid P_{0,T})
=
H(\pi\mid \rho\otimes\mu)
+
\sum_{i,j\in \mathcal I}
\iint
c_{ij}(x,y)
\pi_{ij}(x,y)\,\rho_i(dx)\mu_j(dy)
+
\mathrm{constant}.
\]
Hence, the static Schrödinger problem is equivalent to a standard entropic optimal transport problem with
cost $c_{ij}(x,y)$.

Finally, the factorized form of the optimal coupling can be rewritten accordingly. From
Theorem~\ref{thm:classic}, we have
\[
\pi^{\rho,\mu}_{ij}(dx,dy)
=
\tilde f_i(x)p_{ij}(x,y)\tilde g_j(y)\,
\rho_i(dx)\mu_j(dy).
\]
Substituting
\begin{equation}
\label{eq:change_kernel}
p_{ij}(x,y) = \frac{\rho_{\mathrm{ref},i}(x) r_{ij}(0,x;T,y)}{\rho_i(x)\mu_j(y)}
\end{equation}
and defining
\[
f_i(x)
:=
\tilde f_i(x)\frac{\rho_{\mathrm{ref},i}(x)}{\rho_i(x)},
\qquad
g_j(y)
:=
\frac{\tilde g_j(y)}{\mu_j(y)},
\]
we obtain
\[
\pi^{\rho,\mu}_{ij}(dx,dy)
=
f_i(x)r_{ij}(0,x;T,y)g_j(y)\,
\rho_i(dx)\mu_j(dy).
\]

\paragraph{Sinkhorn iteration.}
Based on the static EOT formulation with the transition kernel $r$ as the reference
kernel, the Schrödinger system can be solved using the classical Sinkhorn iteration
(also known as the Iterative Proportional Fitting Procedure, or IPFP). Given an
arbitrary family of strictly positive initial values
$g^0=(g_j^0)_{j\in\mathcal I}$, the iteration is defined recursively as:
\[
f_i^{n+1}(x)
=
\frac{1}{
\sum_{j\in\mathcal I}
\int_{\mathbb R^d}
r_{ij}(0,x;T,y)g_j^n(y)\,\mu_j(dy)
},
\qquad
i\in\mathcal I,
\]
and
\[
g_j^{n+1}(y)
=
\frac{1}{
\sum_{i\in\mathcal I}
\int_{\mathbb R^d}
r_{ij}(0,x;T,y)f_i^{n+1}(x)\,\rho_i(dx)
},
\qquad
j\in\mathcal I.
\]
Accordingly, after the first half-step update, the intermediate coupling is defined as:
\begin{equation}
\pi_{ij}^{n+1,n}(dx,dy)
=
f_i^{n+1}(x)
r_{ij}(0,x;T,y)
g_j^n(y)\,
\rho_i(dx)\mu_j(dy),
\label{eq:half_step_coupling}
\end{equation}
and upon completing the second half-step update, the full-step coupling is defined as:
\begin{equation}
\pi_{ij}^{n+1,n+1}(dx,dy)
=
f_i^{n+1}(x)
r_{ij}(0,x;T,y)
g_j^{n+1}(y)\,
\rho_i(dx)\mu_j(dy).
\label{eq:full_step_coupling}
\end{equation}

We also introduce the corresponding intermediate marginals generated by the Sinkhorn
iteration. For each $n\ge 0$, define
\begin{equation}
\label{eq:mu_definitions}
\rho^{n,n}(dz_1)
:=
\pi^{n,n}\bigl(dz_1\times(\mathbb R^d\times\mathcal I)\bigr),
\qquad
\mu^{n+1,n}(dz_2)
:=
\pi^{n+1,n}\bigl((\mathbb R^d\times\mathcal I)\times dz_2\bigr).
\end{equation}
Equivalently, in componentwise form,
\[
\frac{d\rho_i^{n,n}}{d\rho_i}(x)
=
f_i^n(x)
\sum_{j\in\mathcal I}
\int_{\mathbb R^d}
r_{ij}(0,x;T,y)g_j^n(y)\,\mu_j(dy),
\qquad
i\in\mathcal I,
\]
and
\[
\frac{d\mu_j^{n+1,n}}{d\mu_j}(y)
=
g_j^n(y)
\sum_{i\in\mathcal I}
\int_{\mathbb R^d}
r_{ij}(0,x;T,y)f_i^{n+1}(x)\,\rho_i(dx),
\qquad
j\in\mathcal I.
\]
Moreover, one has
\[
\pi^{n,n}\in\Pi(\rho^{n,n},\mu),
\qquad
\pi^{n+1,n}\in\Pi(\rho,\mu^{n+1,n}).
\]

\subsection{Notations}
\label{subsec:notation}

\noindent\textbf{Notation 1.} Given \(z=(x,i)\in \mathbb{R}^d\times \mathcal I\) and a function \(f:\mathbb{R}^d\times \mathcal I\to\mathbb{R}\), we may equivalently view \(f\) as a family \((f_i)_{i\in\mathcal I}\) of functions on \(\mathbb{R}^d\), defined by
\[
f_i(x):=f(x,i)=f(z), \qquad x\in\mathbb{R}^d,\ i\in\mathcal I,\ z=(x,i).
\]
Accordingly, we shall use the notations \(f(z)\) and \(f_i(x)\) interchangeably whenever \(z=(x,i)\).

\noindent\textbf{Notation 2.} Extending the above convention, for a function \(f:(\mathbb{R}^d\times \mathcal I)^2\to\mathbb{R}\), we write
\[
f_{ij}(x,y):=f\big((x,i),(y,j)\big), \qquad x,y\in\mathbb{R}^d,\ i,j\in\mathcal I.
\]
Accordingly, we shall use the notations \(f\big((x,i),(y,j)\big)\) and \(f_{ij}(x,y)\) interchangeably.

\noindent\textbf{Notation 3.} For a measure \(\eta\) on
\(\mathbb{R}^d\times\mathcal I\), define its \(i\)-th component by
\[
\eta_i(A)
:=
\eta\bigl(A\times\{i\}\bigr),
\qquad
A\subset\mathbb{R}^d\ \text{Borel},
\qquad
i\in\mathcal I.
\]
Its discrete marginal is denoted by \(\bar\eta\), namely
\begin{equation}
\label{eq:discrete_marginal}
\bar\eta(i)
:=
\eta_i(\mathbb{R}^d).
\end{equation}
Whenever \(\bar\eta(i)>0\), we write \(\hat\eta_i\) for the conditional
probability measure on \(\mathbb{R}^d\) given \(i\), defined by
\begin{equation}
\label{eq:conditional_law}
\hat\eta_i(A)
:=
\frac{\eta_i(A)}{\bar\eta(i)}.
\end{equation}
Thus \(\eta_i=\bar\eta(i)\hat\eta_i\). In particular, we use
\(\rho_i,\bar\rho,\hat\rho_i\) for the initial marginal and
\(\mu_j,\bar\mu,\hat\mu_j\) for the terminal marginal.

\paragraph{Notation for the stability estimate.}
In view of the static reduction of the regime-switching Schrödinger bridge problem discussed above, the stability estimate will first be formulated at the endpoint level. We write
\[
E := \mathbb{R}^d \times \mathcal{I}.
\]
Fix a first marginal $\rho \in \mathcal{P}(E)$. For any terminal marginal $\eta \in \mathcal{P}(E)$, we denote by $\pi^{\rho,\eta}$ the unique optimizer of the static entropic optimal transport problem with marginals $\rho$ and $\eta$, associated with the cost $c$ introduced through the reference transition kernel above. Equivalently,
\[
\pi^{\rho,\eta}
\in
\operatorname*{argmin}_{\pi \in \Pi(\rho,\eta)}
\left\{
H(\pi \mid \rho \otimes \eta)
+
\int_{E \times E} c(z_0,z_T)\,\pi(\mathrm{d}z_0,\mathrm{d}z_T)
\right\}.
\]
In the stability theorem below, the two terminal marginals to be compared will be denoted by $\mu$ and $\nu$, and the corresponding optimal endpoint couplings will be written as $\pi^{\rho,\mu}$ and $\pi^{\rho,\nu}$.

\subsection{Main Results for the Regime Switching Schrödinger Bridge}
\label{subsec:main_result}
We are now ready to state the main results concerning. For this let us group up the main set of assumptions.

In order to be able to tackle our problem we need a suitable version of the Talagrand's transport inequality which classically holds for a measure \(\mu\in\mathcal{P}(\mathbb{R}^d)\) with constant \(C>0\) if
\begin{equation}
\label{eq:t2_euclidean}
W_2^2(\eta,\mu)
\le
C H(\eta\mid\mu),
\qquad
\forall\,\eta\in\mathcal{P}(\mathbb{R}^d).
\end{equation}

The remaining assumptions we impose on the marginal constraints are then as follows.

\begin{assumption}
\label{assum:measures}
The endpoint measures \(\rho, \,\mu \in \mathcal{P}(\mathbb{R}^d)\) satisfy the following conditions.
Let \(\bar\mu\) and \((\hat\mu_i)_{i\in\mathcal I}\) be defined
by \eqref{eq:discrete_marginal} and \eqref{eq:conditional_law}.
\begin{enumerate}
\item \textbf{Absolute continuity.}
For every \(i,j\in\mathcal{I}\) with \(\bar{\rho}(i)>0\) and \(\bar{\mu}(j)>0\), the conditional laws \(\hat{\rho}_i\) and \(\hat{\mu}_j\) are absolutely continuous with respect to the Lebesgue measure on \(\mathbb{R}^d\).

\item \textbf{Compact support.}
The endpoint measures \(\rho\) and \(\mu\) are compactly supported on \(\mathbb{R}^d\times\mathcal{I}\).

\item \textbf{Talagrand \(T_2\) inequality at one endpoint.}
The measure \(\mu\) satisfies the following regime-wise Talagrand \(T_2\)
inequality with constant \(C_{\mu}>0\)
\begin{equation}
\label{eq:t2_regimewise}
W_2^2(\eta,\hat\mu_i)
\le
C_\mu H(\eta\mid\hat\mu_i),
\qquad
\forall\,i\in\mathcal I\ \text{such that}\ \bar\mu(i)>0,
\quad
\forall\,\eta\in\mathcal{P}(\mathbb{R}^d).
\end{equation}
\end{enumerate}
\end{assumption}

We additionally require some regularity on the reference process $P$, which we directly state as assumption on the transition density.

\begin{assumption}
\label{assum:map_c2}
Let the transition kernel of the reference process admit a density $r$ defined by \eqref{eq:trans_density}.
For every pair $i, j \in \mathcal{I}$ with $\bar{\rho}(i)\bar{\mu}(j) > 0$, let the map
\[
(x,y)\longmapsto r_{ij}(0,x;T,y)
\]
be of class $\mathcal{C}^2$ and strictly positive on 
$\operatorname{supp}(\hat{\rho}_i)\times \operatorname{supp}(\hat{\mu}_j)$.
\end{assumption}

As above, we denote by \(c_{ij}\) the cost function associated with the reference kernel, namely
\[
c_{ij}(x,y)=-\log r_{ij}(0,x;T,y).
\]

For later use, we introduce the global constants
\begin{equation}\label{eq:M_constants}
\begin{aligned}
M_0
&:=
\max_{i,j:\,\bar\rho(i)\bar\mu(j)>0}
\ \sup_{(x,y)\in \operatorname{supp}(\hat{\rho}_i)\times \operatorname{supp}(\hat{\mu}_j)}
|c_{ij}(x,y)|, \\
M_1
&:=
\max_{i,j:\,\bar\rho(i)\bar\mu(j)>0}
\ \sup_{(x,y)\in \operatorname{supp}(\hat{\rho}_i)\times \operatorname{supp}(\hat{\mu}_j)}
\|\nabla c_{ij}(x,y)\|, \\
M_2
&:=
\max_{i,j:\,\bar\rho(i)\bar\mu(j)>0}
\ \sup_{(x,y)\in \operatorname{supp}(\hat{\rho}_i)\times \operatorname{supp}(\hat{\mu}_j)}
\|\nabla^2 c_{ij}(x,y)\|.
\end{aligned}
\end{equation}

By Assumptions~~\ref{assum:map_c2}, all these constants are finite. Here \(\nabla c_{ij}\) and \(\nabla^2 c_{ij}\) denote respectively the full gradient and the full Hessian with respect to the variable \((x,y)\).

\begin{remark}
The constants \(M_1\) and \(M_2\) defined above control derivatives
with respect to both variables \(x\) and \(y\). This symmetric
formulation is convenient, but stronger than what is needed below.
More precisely, if the \(T_2\) condition in
Assumption~\ref{assum:measures} is imposed on the terminal marginal
\(\mu\), only the bounds on \(\nabla_y c_{ij}\) and
\(\nabla^2_{yy} c_{ij}\) are required. If instead the \(T_2\) condition is
imposed on the initial marginal \(\rho\), then, by the transposition
argument in Remark~\ref{rem:transposition}, it is enough to use the
corresponding bounds on \(\nabla_x c_{ij}\) and
\(\nabla^2_{xx} c_{ij}\).
Since we are not aiming to optimize on the constants here, we decided to keep it that way.
\end{remark}

The above assumptions are satisfied in a variety of cases as the following examples show.

\begin{example}
\label{ex:t2_compact}
Assumption~\ref{assum:measures} is satisfied, for instance, if one of the two endpoint families consists of uniformly strongly log-concave measures supported on compact convex sets. We state the condition for the initial marginal; the analogous condition on the terminal marginal yields the same conclusion.

For each $i\in\mathcal I$ such that $\rho(\mathbb R^d\times\{i\})>0$, let
\[
K_i:=\operatorname{supp}(\hat{\rho}_i).
\]
Assume that $K_i$ is compact and convex with nonempty interior, and that
\[
\hat{\rho}_i(\mathrm{d}x)
=
\frac{1}{Z_i}e^{-V_i(x)}\mathbf{1}_{K_i}(x)\,\mathrm{d}x,
\]
where $V_i\in C^2(K_i^\circ)$ and $ Z_i $ is the normalisation constant. Suppose that there exists $\kappa>0$, such that
\[
\nabla^2V_i(x)\succeq\kappa I_d,
\qquad
x\in K_i^\circ.
\]
Then, for every $\eta\in\mathcal P(\mathbb R^d)$,
\[
W_2^2(\eta,\hat{\rho}_i)
\leq
\frac{2}{\kappa}H(\eta\mid\hat{\rho}_i).
\]
Hence, the family $(\hat{\rho}_i)_i$ satisfies a Talagrand $T_2$ inequality, with constant $2/\kappa$.

Indeed, since $K_i$ is convex, the above Hessian bound implies the $\kappa$-displacement convexity of $H(\,\cdot\,\mid\hat{\rho}_i)$. Applying this estimate along the Wasserstein geodesic from $\hat{\rho}_i$ to $\eta$, and using the nonnegativity of relative entropy, gives the stated $T_2$ inequality.

Such log-concave assumptions are standard sufficient conditions for Talagrand inequalities and arise naturally in the analysis of entropic optimal transport; see, for instance, \cite{ChiariniConfortiGrecoTamanini2024SemiconcStabConv,ConfortiDurmusGreco2023SinkhornContraction}.
\end{example}

\begin{example}
\label{ex:uniform_t2}
Assumption~\ref{assum:measures} is also satisfied if one of the two endpoint
families has compact convex supports and admits densities uniformly bounded
from above and below. More precisely, it is enough that either the following
condition holds for the initial marginal, or the analogous condition holds for
the terminal marginal.

For the initial marginal, suppose that for each \(i\in\mathcal I\) such that
\(\bar\rho(i)>0\), the support
\(K_i:=\operatorname{supp}(\hat{\rho}_i)\) is compact and convex, and
\(\hat{\rho}_i\) admits a density \(q_i\) on \(K_i\) satisfying
\[
\hat{\rho}_i(\mathrm{d}x)
=
q_i(x)\mathbf{1}_{K_i}(x)\,\mathrm{d}x,
\qquad
0<m\leq q_i(x)\leq M<\infty,
\qquad
x\in K_i,
\]
where \(m\) and \(M\) are independent of \(i\).

Then the family \(\bigl(\hat{\rho}_i\bigr)_{i\in\mathcal I}\) satisfies a
Talagrand \(T_2\) inequality. The same conclusion holds if the above
assumptions are imposed instead on the terminal conditional laws
\(\bigl(\hat{\mu}_j\bigr)_{j\in\mathcal I}\).

This hypothesis is stronger than assuming that the family
\(\bigl(\hat{\rho}_i\bigr)_{i\in\mathcal I}\) satisfies a Talagrand
\(T_2\) inequality. Indeed, it implies a logarithmic Sobolev inequality (LSI) and
hence a Talagrand \(T_2\) inequality; cf.\
\cite[Proposition~5.3 and the discussion therein]{BobkovLedoux2000} and
\cite[Proposition~5.1.6]{BakryGentilLedoux2013}.
\end{example}

\begin{example}
\label{ex:c2}
Assumption~\ref{assum:map_c2} holds in particular for constant-coefficient reference processes. Indeed, assume that \(b_i(t,x)\equiv b_i\), \(\sigma_i(t,x)\equiv \sigma_i I_d\), and \(\lambda_{ij}(t,x)\equiv \lambda_{ij}\) for all \(i,j\in\mathcal I\), where \(\sigma_i>0\) for every \(i\in\mathcal I\) and \(\lambda_{ij}>0\) for every \(i\neq j\). Then, for every \(T>0\), the transition density \(r_{ij}(0,x;T,y)\) is strictly positive and of class \(C^2\) in \((x,y)\). The proof is given in Appendix~\ref{app:c2}.
\end{example}

We are now in place to state the main stability result for the regime switching Schrödinger problem.

\begin{theorem}
[Stability in relative entropy]
\label{thm:stable}
Let Assumptions \ref{assum:reference_marginals}, \ref{assum:measures} and~\ref{assum:map_c2} hold for \((\rho,\mu) \in \mathcal{P}(\mathbb{R}^d)^2\) and the reference process $P \in \mathcal{P}(\Omega_T)$.

Let \(M_0\), \(M_1\), and \(M_2\) be defined in \eqref{eq:M_constants}, and set
\[
B
:=
\max\left\{
M_2+\frac{M_1^2}{2},
\,4M_0
\right\}.
\]
Then, for every \(\nu\in\mathcal P(\mathbb R^d\times\mathcal I)\) such that
\begin{equation}
\label{eq:mutual_finite_entropy}
H(\mu\mid\nu)+H(\nu\mid\mu)<+\infty,
\end{equation}
one has
\[
\begin{aligned}
\mathcal H\left(Q^{\rho,\mu}\mid Q^{\rho,\nu}\right)
&=
H\left(\pi^{\rho,\mu}\mid \pi^{\rho,\nu}\right) \\
&\le
H(\mu\mid\nu)
+
BC_\mu H(\nu\mid\mu)
+
B\|\bar\mu-\bar\nu\|_{\mathrm{TV}} \\
&\le
H(\mu\mid\nu)
+
BC_\mu H(\nu\mid\mu)
+
B\|\mu-\nu\|_{\mathrm{TV}}.
\end{aligned}
\]
\end{theorem}

\begin{remark}
By \eqref{eq:mutual_finite_entropy}, one has $\mu\sim\nu$. Consequently, the first two conditions in Assumption~\ref{assum:measures} also hold for the pair $(\rho,\nu)$. In particular, $\operatorname{supp}\mu=\operatorname{supp}\nu$. Hence, Assumption~\ref{assum:map_c2} imposed for $(\rho,\mu)$ also holds for $(\rho,\nu)$, with the same constants $M_0$, $M_1$, and $M_2$. 
 The $T_2$ condition in Assumption~\ref{assum:measures} need not hold for $\nu$ and is not used in the proof.
\end{remark}

\begin{theorem}[Exponential convergence in relative entropy]
\label{thm:exp}
Let Assumptions~ \ref{assum:reference_marginals}, \ref{assum:measures} and \ref{assum:map_c2} hold for \((\rho,\mu) \in \mathcal{P}(\mathbb{R}^d)^2\) and the reference process $P \in \mathcal{P}(\Omega_T)$. Let $\pi^{\rho,\mu}$ be the unique optimal coupling, and let $\pi^{n+1,n}$ and $\pi^{n+1,n+1}$ be, respectively, the half-step and full-step couplings defined in \eqref{eq:half_step_coupling} and \eqref{eq:full_step_coupling}. Assume that $ H(\pi^{\rho,\mu} \mid \pi^{1,0}) < \infty$, $H(\pi^{\rho,\mu} \mid \pi^{1,1}) < \infty$. Then there exist constants $C > 0$ and $\theta \in (0, 1)$  depending on $ \rho, \mu, P $ and $ \pi^{1,0} $, $ \pi^{1,1} $ such that, for all $n \geq 0$,
\[
H(\pi^{\rho,\mu} \mid \pi^{n+1,n}) \leq C \theta^n,
\qquad
H(\pi^{\rho,\mu} \mid \pi^{n+1,n+1}) \leq C \theta^n .
\]
\end{theorem}

For a precise description of the dependence of $C$ and $\theta$ on $ \rho, \mu, P $ and the initial data see the proof of the theorem in Section \ref{sec:proof-RS}.
\begin{remark}
\label{remark:lift}
By the entropy decomposition above, each intermediate Sinkhorn coupling, in particular $\pi^{n+1,n}$ and $\pi^{n+1,n+1}$, can be canonically lifted back to a path measure by keeping the reference conditional bridge $P(\,\cdot\,|Z_0,Z_T)$ unchanged and replacing only the endpoint law. In particular, one obtains corresponding path measures $Q^{n+1,n}$ and $Q^{n+1,n+1}$ such that
\[
\mathcal H(Q^{\rho,\mu} \mid Q^{n+1,n}) = H(\pi^{\rho,\mu} \mid \pi^{n+1,n}), 
\qquad
\mathcal H(Q^{\rho,\mu} \mid Q^{n+1,n+1}) = H(\pi^{\rho,\mu} \mid \pi^{n+1,n+1}) .
\]
Moreover, since these lifted measures still have the same endpoint factorization as in Theorem~\ref{thm:classic}, the reconstruction procedure in Theorem~\ref{thm:classic}(ii)--(iii) applies verbatim. Therefore, the exponential convergence established in Theorem~\ref{thm:exp} for the marginal Sinkhorn iterates immediately yields exponential convergence of the reconstructed path measures toward the optimal Schr\"odinger bridge.
\end{remark}

\begin{remark}
\label{rem:transposition}
In the proof of Theorem \ref{thm:stable} and \ref{thm:exp}, we treat only the case where \(\mu\) satisfies the
Talagrand \(T_2\) inequality \eqref{eq:t2_regimewise}. Suppose instead that
this condition holds for the initial measure \(\rho\), and we vary \(\rho\) instead of \(\mu\). While the problem may not be symmetric in the marginals, it is still sufficient to flip the marginals and costs. At the level of the static EOT problem,
define the transposed cost by
\[
c^\top(w,z):=c(z,w),
\qquad\text{that is,}\qquad
c^\top_{ji}(y,x):=c_{ij}(x,y).
\]
Then, for every \(\pi\in\Pi(\rho,\mu)\),
\[
\begin{aligned}
&H(\pi\mid\rho\otimes\mu)
+\int_{E\times E} c(z,w)\,\pi(\mathrm{d}z,\mathrm{d}w) \\
&\qquad =
H(\pi^\top\mid\mu\otimes\rho)
+\int_{E\times E} c^\top(w,z)\,
\pi^\top(\mathrm{d}w,\mathrm{d}z).
\end{aligned}
\]
Consequently, transposition identifies the static EOT problem with marginals
\((\rho,\mu)\) and cost \(c\) with the problem with marginals \((\mu,\rho)\)
and cost \(c^\top\). Under this identification, the two Sinkhorn updates are
exchanged. Thus, the argument for the case where the second marginal satisfies
the \(T_2\) condition applies to the transposed problem when the condition is
imposed on \(\rho\).

This transposition concerns only the static EOT formulation but
does not yield a symmetry of the dynamic Schrödinger bridge. Nevertheless, the lifted Sinkhorn iterates satisfy
\[
H\!\left(Q^{\rho,\mu}\mid Q^{n+1,n}\right)
=
H\!\left(\pi^{\rho,\mu}\mid\pi^{n+1,n}\right),
\]
as detailed in Remark~\ref{remark:lift}.
Hence, their exponential convergence rate on
path space is exactly the rate obtained for the static EOT iterates, so it is
enough to treat the case where \(\mu\) satisfies the \(T_2\) condition.
\end{remark}

\subsection{Main Results on the Partially Observed Problem}
\label{subsec:partial}

We now consider a partially observed variant of the regime-switching Schr\"odinger bridge problem, in which the terminal regime is not observed. More precisely, the initial marginal is prescribed on the full hybrid space $\mathbb{R}^d \times \mathcal{I}$, while at time $T$ only the marginal distribution of the continuous component $X_T$ is prescribed.

To the best of our knowledge, this partially observed problem has not been considered in the literature.

Let
\[
\rho \in \mathcal P(\mathbb R^d \times \mathcal{I}),
\qquad
\mu_{\mathrm p} \in \mathcal P(\mathbb R^d),
\]
denote the observed marginals
The partially observed regime-switching Schr\"odinger bridge problem is defined by
\begin{equation*}
\inf_{Q \ll P}
\left\{
\mathcal{H}(Q \mid P)
:
Q \circ Z_0^{-1} = \rho,
\quad
Q \circ X_T^{-1} = \mu_{\mathrm p}
\right\}.
\end{equation*}

By the entropy decomposition formula, for every admissible path measure \(Q \ll P\), if
\[
Q_{0,T}:=Q\circ (Z_0,Z_T)^{-1},
\qquad
P_{0,T}:=P\circ (Z_0,Z_T)^{-1},
\]
then
\[
\mathcal H(Q\mid P)
=
H(Q_{0,T}\mid P_{0,T})
+
\int
\mathcal H\bigl(Q(\cdot\mid z_0,z_T)\mid P(\cdot\mid z_0,z_T)\bigr)
\,Q_{0,T}(\mathrm dz_0,\mathrm dz_T).
\]
Therefore, as before
\begin{align}
&\inf_{Q\ll P}
\Bigl\{
H(Q\mid P):
Q\circ Z_0^{-1}=\rho,\;
Q\circ X_T^{-1}=\mu_{\mathrm p}
\Bigr\}
\\
\label{eq:eot-partial}
&\qquad
=
\inf_{\pi\in\mathcal P((\mathbb R^d\times\mathcal I)^2)}
\Bigl\{
H(\pi\mid P_{0,T}):
\pi\circ Z_0^{-1}=\rho,\;
\pi\circ X_T^{-1}=\mu_{\mathrm p}
\Bigr\}.
\end{align}

Thus the partially observed Schrödinger bridge problem is reduced to a static problem with partial observation.

\begin{assumption}
\label{assum:reference_marginals_partial}
Let $\bar P_{0,T}:=P\circ (Z_0,X_T)^{-1} $
be the projected endpoint law of the reference process. Assume that
\[
H(\rho\otimes\mu_{\mathrm p}\mid \bar P_{0,T})<\infty \quad \text{ and } \quad
\bar P_{0,T}\big|_{\operatorname{supp}(\rho\otimes\mu_{\mathrm p})}
\ll \rho\otimes\mu_{\mathrm p}.
\]
\end{assumption}

Define \(\bar p(z_0,y):=\frac{d\left(\left.\bar P_{0,T}\right|_{\operatorname{supp}(\rho\otimes\mu_{\mathrm p})}\right)}{d(\rho\otimes\mu_{\mathrm p})}(z_0,y)\). In componentwise notation, for \(z_0=(x,i)\) with \(i\in\mathcal I\), write \(\bar p_i(x,y):=\bar p((x,i),y)\), so that
\begin{equation}
\bar p_i(x,y)
:=
\frac{
d\!\left(
\bar P_{0,T}\big|_{\operatorname{supp}(\rho\otimes\mu_{\mathrm p})}
\right)
}{
d(\rho\otimes\mu_{\mathrm p})
}\bigl((x,i),y\bigr).
\label{eq:bar_p}
\end{equation}

As in the fully observed case \eqref{eq:change_kernel}, the EOT problem with kernel
$\bar p$ is equivalent to the EOT problem
with the transition kernel $\bar r$:
\[
\bar r_i(0,x;T,y)dy
:=
P(X_T\in dy\mid X_0=x,I_0=i),
\qquad
i\in\mathcal I.
\]
Then the projected transition density satisfies
\begin{equation}
\label{eq:projected_transition}
\bar r_i(0,x;T,y)
=
\sum_{j\in\mathcal I}
r_{ij}(0,x;T,y).
\end{equation}

For each $i\in I$, define
\[
\bar p_i(x,y)
:=
\frac{
d\!\left(
\bar P_{0,T}\big|_{\operatorname{supp}(\rho\otimes\mu_{\mathrm p})}
\right)
}{
d(\rho\otimes\mu_{\mathrm p})
}\big((x,i),y\big).
\]

\begin{theorem}[Characterization of the optimal control under partial terminal observation]
\label{thm:partial_observation}
Under Assumption \ref{assum:reference_marginals_partial} the partially observed static regime-switching Schr\"odinger bridge problem \eqref{eq:eot-partial} admits
a unique optimal solution $\pi^{\rho,\mu_{\mathrm p}}\ll P_{0, T}$, characterized as follows.
There exist nonnegative measurable functions $(\tilde f_i)_{i\in\mathcal I}$ and $\tilde g$, unique up to the transformation
\[
(\tilde f,\tilde g)
\sim
(\lambda\tilde f,\lambda^{-1}\tilde g),
\qquad
\lambda>0,
\]
such that
\begin{equation}
\label{eq:partial_path_factorization}
\frac{\mathrm d\pi^{\rho,\mu_{\mathrm p}}}{\mathrm dP_{0,T}}(z_0,z_T)
=
\tilde f_{i_0}(x_0)\tilde g(x_T) \quad \text{for }\pi^{\rho,\mu}\text{-a.e. } (z_0,z_T)=((x_0,i_0),(x_T,i_T)).
\end{equation}
Moreover, $(\tilde f_i)_{i\in\mathcal I}$ and $\tilde g$ solve the projected Schr\"odinger system
\begin{equation}
\label{eq:partial_schrodinger_system_reference}
\left\{
\begin{aligned}
1
&=
\tilde f_i(x)
\int_{\mathbb R^d}
\bar p_i(x,y)\tilde g(y)\,\mu_{\mathrm p}(\mathrm dy),
&&
\rho_i\text{-a.e. }x,
\quad i\in\mathcal I,
\\
1
&=
\tilde g(y)
\sum_{i\in\mathcal I}
\int_{\mathbb R^d}
\bar p_i(x,y)\tilde f_i(x)\,\rho_i(\mathrm dx),
&&
\mu_{\mathrm p}\text{-a.e. }y.
\end{aligned}
\right.
\end{equation}

Equivalently, the optimal projected endpoint coupling $\gamma^{\rho,\mu_{\mathrm p}}
:=
(\mathrm{id}_E \times \mathrm{pr}_{\mathbb{R}^d})_{\#} \pi^{\rho,\mu_{\mathrm p}}$, where $ \mathrm{pr}_{\mathbb{R}^d}(y, i) = y $, is given componentwise by
\[
\gamma_i^{\rho,\mu_{\mathrm p}}(\mathrm dx,\mathrm dy)
=
\tilde f_i(x)\bar p_i(x,y)\tilde g(y)\,
\rho_i(\mathrm dx)\mu_{\mathrm p}(\mathrm dy),
\qquad
i\in\mathcal I.
\]
\end{theorem}

\begin{remark}
As for the fully observed case, we can give an expression for optimal $(b^*,\sigma^*,\lambda^*)$ inducing the path law $Q^{\rho,\mu_{\mathrm p}}$.
Let the function $h=(h_i)_{i\in\mathcal I}$ be defined by \[
h_i(t,x)
:=
\mathbb E_P\!\left[
\tilde g(X_T)
\mid
X_t=x,\ I_t=i
\right],
\qquad
(t,x,i)\in[0,T]\times\mathbb R^d\times\mathcal I.
\]
Let Assumptions (C), (G) and Assumption (H) of \cite{zlotchevski2025schrodingerbridgeproblemjump} for $h$ be satisfied.
Then the optimal control triplet $(b^*,\sigma^*,\lambda^*)$ inducing the path law $Q^{\rho,\mu_{\mathrm p}}$ is given by
\[
\begin{aligned}
\sigma_i^*(t,x)
&=
\sigma_i(t,x),
\\
b_i^*(t,x)
&=
b_i(t,x)+a_i(t,x)\nabla_x\log h_i(t,x),
\\
\lambda_{ij}^*(t,x)
&=
\lambda_{ij}(t,x)\frac{h_j(t,x)}{h_i(t,x)},
\qquad
j\neq i,
\end{aligned}
\]
where $a_i(t,x)=\sigma_i(t,x)\sigma_i(t,x)^\top$, and
the function $h$ satisfies
\[
\left\{
\begin{aligned}
&\partial_t h_i(t,x)
+b_i(t,x)\cdot\nabla_x h_i(t,x)
+\frac12\operatorname{Tr}\!\bigl(a_i(t,x)\nabla_x^2h_i(t,x)\bigr)
\\
&\qquad
+\sum_{j\neq i}
\lambda_{ij}(t,x)\bigl(h_j(t,x)-h_i(t,x)\bigr)
=0,
&&
(t,x,i)\in[0,T)\times\mathbb R^d\times\mathcal I,
\\
&h_i(T,x)=\tilde g(x),
&&
x\in\mathbb R^d,
\quad i\in\mathcal I.
\end{aligned}
\right.
\]

\end{remark}

\begin{assumption}
\label{assum:po_measures}
The endpoint measures $\rho\in\mathcal P(\mathbb R^d\times\mathcal I)$ and $\mu_{\mathrm p}\in\mathcal P(\mathbb R^d)$ satisfy the following conditions.
\begin{enumerate}
\item \textbf{Absolute continuity.}
For every $i\in\mathcal I$ with $\bar\rho(i)>0$, the conditional law $\hat\rho_i$ and the measure $\mu_{\mathrm p}$ are absolutely continuous with respect to Lebesgue measure on $\mathbb R^d$.

\item \textbf{Compact support.}
The measure $\rho$ is compactly supported on $\mathbb R^d\times\mathcal I$, and $\mu_{\mathrm p}$ is compactly supported on $\mathbb R^d$.

\item \textbf{Talagrand $T_2$ inequality at the observed terminal endpoint.}
The measure $\mu_{\mathrm p}$ satisfies the Talagrand $T_2$ inequality \eqref{eq:t2_euclidean} with constant $C_{\mu_{\mathrm p}}>0$.
\end{enumerate}
\end{assumption}

\begin{assumption}
\label{assum:po_map_c2}
Let the transition kernel of the reference process admit a density $r$ defined by \eqref{eq:trans_density}.
For every $i\in\mathcal I$ with $\bar\rho(i)>0$, let the map
\[
(x,y)
\longmapsto
\bar r_i(0,x;T,y),
\]
defined by \eqref{eq:projected_transition}, be of class $C^2$ and strictly positive on
\[
\operatorname{supp}(\hat\rho_i)
\times
\operatorname{supp}(\mu_{\mathrm p}).
\]
\end{assumption}

As above, define the projected cost by
\[
\bar c_i(x,y)
:=
-\log\bar r_i(0,x;T,y).
\]
For later use, set
\begin{equation}
\label{eq:partial_M_constants}
\begin{aligned}
\bar M_0
&:=
\max_{i:\,\bar\rho(i)>0}
\ \sup_{(x,y)\in\operatorname{supp}(\hat\rho_i)\times\operatorname{supp}(\mu_{\mathrm p})}
|\bar c_i(x,y)|,
\\
\bar M_1
&:=
\max_{i:\,\bar\rho(i)>0}
\ \sup_{(x,y)\in\operatorname{supp}(\hat\rho_i)\times\operatorname{supp}(\mu_{\mathrm p})}
\|\nabla\bar c_i(x,y)\|,
\\
\bar M_2
&:=
\max_{i:\,\bar\rho(i)>0}
\ \sup_{(x,y)\in\operatorname{supp}(\hat\rho_i)\times\operatorname{supp}(\mu_{\mathrm p})}
\|\nabla^2\bar c_i(x,y)\|.
\end{aligned}
\end{equation}
By Assumption~\ref{assum:po_map_c2}, these constants are finite. Here $\nabla\bar c_i$ and $\nabla^2\bar c_i$ denote the full gradient and full Hessian with respect to $(x,y)$.

In the static EOT formulation, it is convenient to work with the projected kernel $\bar r$. Under Assumption~\ref{assum:po_measures}, writing $\rho_i(x)$ and $\mu_{\mathrm p}(y)$ for the corresponding Lebesgue densities, \eqref{eq:bar_p} becomes
\begin{equation}
\label{eq:partial_change_kernel}
\bar p_i(x,y)
=
\frac{
\rho_{\mathrm{ref},i}(x)\bar r_i(0,x;T,y)
}{
\rho_i(x)\mu_{\mathrm p}(y)
}.
\end{equation}
Thus, with
\[
f_i(x)
:=
\tilde f_i(x)\frac{\rho_{\mathrm{ref},i}(x)}{\rho_i(x)},
\qquad
g(y)
:=
\frac{\tilde g(y)}{\mu_{\mathrm p}(y)},
\]
the optimal projected endpoint coupling becomes
\[
\gamma_i^{\rho,\mu_{\mathrm p}}(\mathrm dx,\mathrm dy)
=
f_i(x)\bar r_i(0,x;T,y)g(y)\,
\rho_i(\mathrm dx)\mu_{\mathrm p}(\mathrm dy),
\qquad
i\in\mathcal I.
\]
Equivalently, $(f_i)_{i\in\mathcal I}$ and $g$ solve
\begin{equation}
\label{eq:partial_schrodinger_system}
\left\{
\begin{aligned}
1
&=
f_i(x)
\int_{\mathbb R^d}
\bar r_i(0,x;T,y)g(y)\,\mu_{\mathrm p}(\mathrm dy),
&&
i\in\mathcal I,
\\
1
&=
g(y)
\sum_{i\in\mathcal I}
\int_{\mathbb R^d}
\bar r_i(0,x;T,y)f_i(x)\,\rho_i(\mathrm dx).
\end{aligned}
\right.
\end{equation}

Given a strictly positive initial function $g^0$, the projected Sinkhorn iteration is
\begin{equation}
\label{eq:partial_sinkhorn_update}
\left\{
\begin{aligned}
f_i^{n+1}(x)
&=
\frac{1}{
\displaystyle
\int_{\mathbb R^d}
\bar r_i(0,x;T,y)g^n(y)\,\mu_{\mathrm p}(\mathrm dy)
},
&&
i\in\mathcal I,
\\
g^{n+1}(y)
&=
\frac{1}{
\displaystyle
\sum_{i\in\mathcal I}
\int_{\mathbb R^d}
\bar r_i(0,x;T,y)f_i^{n+1}(x)\,\rho_i(\mathrm dx)
}.
\end{aligned}
\right.
\end{equation}
The associated half-step and full-step couplings are
\[
\begin{aligned}
\gamma_i^{n+1,n}(\mathrm dx,\mathrm dy)
&=
f_i^{n+1}(x)\bar r_i(0,x;T,y)g^n(y)\,
\rho_i(\mathrm dx)\mu_{\mathrm p}(\mathrm dy),
\\
\gamma_i^{n+1,n+1}(\mathrm dx,\mathrm dy)
&=
f_i^{n+1}(x)\bar r_i(0,x;T,y)g^{n+1}(y)\,
\rho_i(\mathrm dx)\mu_{\mathrm p}(\mathrm dy).
\end{aligned}
\]

\paragraph{Notation for the partially observed stability estimate.}
Fix $\rho\in\mathcal P(\mathbb R^d\times\mathcal I)$. For any $\eta_{\mathrm p}\in\mathcal P(\mathbb R^d)$, let $\gamma^{\rho,\eta_{\mathrm p}}$ denote the unique optimizer of the projected static EOT problem
\begin{equation}
\label{eq:partial_static_eot}
\gamma^{\rho,\eta_{\mathrm p}}
\in
\operatorname*{argmin}_{\gamma\in\Pi(\rho,\eta_{\mathrm p})}
\left\{
H(\gamma\mid\rho\otimes\eta_{\mathrm p})
+
\int_{(\mathbb R^d\times\mathcal I)\times\mathbb R^d}
\bar c((x,i),y)\,\gamma(\mathrm d(x,i),\mathrm dy)
\right\},
\end{equation}
where $\bar c((x,i),y):=\bar c_i(x,y)$. The corresponding path-space measure is obtained by keeping the conditional law of the reference process given $(Z_0,X_T)$ unchanged:
\[
Q^{\rho,\eta_{\mathrm p}}(\mathrm d\omega)
=
P(\mathrm d\omega\mid Z_0,X_T)
\,\gamma^{\rho,\eta_{\mathrm p}}(\mathrm dZ_0,\mathrm dX_T).
\]
Equivalently,
\[
Q^{\rho,\eta_{\mathrm p}}\circ(Z_0,X_T)^{-1}
=
\gamma^{\rho,\eta_{\mathrm p}},
\qquad
Q^{\rho,\eta_{\mathrm p}}(\,\cdot\mid Z_0,X_T)
=
P(\,\cdot\mid Z_0,X_T).
\]
Hence, for any $\mu_{\mathrm p},\nu_{\mathrm p}\in\mathcal P(\mathbb R^d)$,
\begin{equation}
\label{eq:partial_path_endpoint_entropy}
\mathcal H\left(Q^{\rho,\mu_{\mathrm p}}\mid Q^{\rho,\nu_{\mathrm p}}\right)
=
H\left(\gamma^{\rho,\mu_{\mathrm p}}\mid\gamma^{\rho,\nu_{\mathrm p}}\right).
\end{equation}

\begin{theorem}[Stability in relative entropy under partial terminal observation]
\label{thm:stability_partial}
Let Assumptions~ \ref{assum:reference_marginals_partial}, \ref{assum:po_measures} and~\ref{assum:po_map_c2} hold for the pair $(\rho,\mu_{\mathrm p})$. Let $\bar M_0$, $\bar M_1$, and $\bar M_2$ be defined in \eqref{eq:partial_M_constants}, and set
\[
\bar B
:=
\max\left\{
\bar M_2+\frac{\bar M_1^2}{2},
\,4\bar M_0
\right\}.
\]
Then, for every $\nu_{\mathrm p}\in\mathcal P(\mathbb R^d)$ such that
\begin{equation}
\label{eq:partial_mutual_finite_entropy}
H(\mu_{\mathrm p}\mid\nu_{\mathrm p})
+
H(\nu_{\mathrm p}\mid\mu_{\mathrm p})
<+\infty,
\end{equation}
one has
\[
\begin{aligned}
\mathcal H\left(Q^{\rho,\mu_{\mathrm p}}\mid Q^{\rho,\nu_{\mathrm p}}\right)
&=
H\left(\gamma^{\rho,\mu_{\mathrm p}}\mid\gamma^{\rho,\nu_{\mathrm p}}\right)
\\
&\le
H(\mu_{\mathrm p}\mid\nu_{\mathrm p})
+
\bar B C_{\mu_{\mathrm p}}H(\nu_{\mathrm p}\mid\mu_{\mathrm p}).
\end{aligned}
\]
\end{theorem}

\begin{theorem}[Exponential convergence under partial terminal observation]
\label{thm:po_exp_conv}
Let Assumptions~  \ref{assum:reference_marginals_partial}, \ref{assum:po_measures} and~\ref{assum:po_map_c2} hold. Let $\gamma^{\rho,\mu_{\mathrm p}}$ be the optimal coupling of the projected static EOT problem, and let $\gamma^{n+1,n}$ and $\gamma^{n+1,n+1}$ be the half-step and full-step couplings generated by \eqref{eq:partial_sinkhorn_update}. Assume that $ H(\gamma^{\rho,\mu_{\mathrm p}} \mid \gamma^{1,0}) < \infty$, $H(\gamma^{\rho,\mu_{\mathrm p}} \mid \gamma^{1,1}) < \infty$. Then there exist constants $A>0$ and $\theta\in(0,1)$ such that, for all $n\ge0$,
\[
H(\gamma^{\rho,\mu_{\mathrm p}}\mid\gamma^{n+1,n})
\le A\theta^n,
\qquad
H(\gamma^{\rho,\mu_{\mathrm p}}\mid\gamma^{n+1,n+1})
\le A\theta^n.
\]

Moreover, let $Q^{n+1,n}$ and $Q^{n+1,n+1}$ be obtained by lifting $\gamma^{n+1,n}$ and $\gamma^{n+1,n+1}$ while keeping $P(\,\cdot\mid Z_0,X_T)$ unchanged. Then
\[
\begin{aligned}
\mathcal H(Q^{\rho,\mu_{\mathrm p}}\mid Q^{n+1,n})
&=
H(\gamma^{\rho,\mu_{\mathrm p}}\mid\gamma^{n+1,n}),
\\
\mathcal H(Q^{\rho,\mu_{\mathrm p}}\mid Q^{n+1,n+1})
&=
H(\gamma^{\rho,\mu_{\mathrm p}}\mid\gamma^{n+1,n+1}).
\end{aligned}
\]
Consequently,
\[
\mathcal H(Q^{\rho,\mu_{\mathrm p}}\mid Q^{n+1,n})
\le A\theta^n,
\qquad
\mathcal H(Q^{\rho,\mu_{\mathrm p}}\mid Q^{n+1,n+1})
\le A\theta^n.
\]
\end{theorem}

\section{Proofs of Theorems~\ref{thm:stable} and~\ref{thm:exp}}\label{sec:proof-RS}
Before entering the proof, we switch from the multiplicative Schrödinger potentials used in Section~\ref{sec:RS} to
their logarithmic counterparts, which are more convenient for the analysis below. Recall that the optimal coupling was written as
\begin{equation}
\pi^{\rho,\mu}_{ij}(dx,dy)
=
f_i(x)r_{ij}(0,x;T,y)g_j(y)\,
\rho_i(dx)\mu_j(dy).
\label{eq:optimal_coupling_multiplicative}
\end{equation}
where
\[
c_{ij}(x,y):=-\log r_{ij}(0,x;T,y).
\]
We now introduce the logarithmic potentials
\[
\varphi_i(x):=-\log f_i(x),
\qquad
\psi_j(y):=-\log g_j(y),
\]
and similarly, for the Sinkhorn iterates,
\[
\varphi_i^n(x):=-\log f_i^n(x),
\qquad
\psi_j^n(y):=-\log g_j^n(y).
\]
With this notation, the optimal coupling can be rewritten as
\[
\pi^{\rho,\mu}_{ij}(dx,dy)
=
\exp\bigl(-\varphi_i(x)-\psi_j(y)-c_{ij}(x,y)\bigr)
\rho_i(dx)\mu_j(dy).
\]
Likewise, the half-step and full-step Sinkhorn couplings become
\[
\pi_{ij}^{n+1,n}(dx,dy)
=
\exp\bigl(-\varphi_i^{n+1}(x)-\psi_j^n(y)-c_{ij}(x,y)\bigr)
\rho_i(dx)\mu_j(dy),
\]
and
\[
\pi_{ij}^{n+1,n+1}(dx,dy)
=
\exp\bigl(-\varphi_i^{n+1}(x)-\psi_j^{n+1}(y)-c_{ij}(x,y)\bigr)
\rho_i(dx)\mu_j(dy).
\]
In terms of these logarithmic potentials, the Sinkhorn iteration reads
\begin{equation}
\label{eq:sinkhorn_update}
\begin{aligned}
\varphi_i^{n+1}(x)
&= \log \! \left( \sum_{j \in \mathcal{I}} \int_{\mathbb{R}^d}
\exp \bigl( -c_{ij}(x, y) - \psi_j^n(y) \bigr)
\, \mu_j(\mathrm{d}y) \right),
\quad i \in \mathcal{I}, \\
\psi_j^{n+1}(y)
&= \log \! \left( \sum_{i \in \mathcal{I}} \int_{\mathbb{R}^d}
\exp \bigl( -c_{ij}(x, y) - \varphi_i^{n+1}(x) \bigr)
\, \rho_i(\mathrm{d}x) \right),
\quad j \in \mathcal{I}.
\end{aligned}
\end{equation}
Equivalently, we may write the Sinkhorn iteration directly on the product space
$E:=\mathbb R^d\times\mathcal I$. For $z=(x,i)\in E$ and $w=(y,j)\in E$, set
\[
c(z,w):=c_{ij}(x,y).
\]
The logarithmic Sinkhorn iteration can then be written in the compact form
\[
\varphi^{n+1}(z)
=
\log\int_E
\exp\bigl(-c(z,w)-\psi^n(w)\bigr)\,\mu(dw),
\qquad
z\in E,
\]
\[
\psi^{n+1}(w)
=
\log\int_E
\exp\bigl(-c(z,w)-\varphi^{n+1}(z)\bigr)\,\rho(dz),
\qquad
w\in E.
\]
Both the componentwise notation and the compact notation on $E$ will be used below, depending on which one is more convenient in each argument.

To analyze the stability estimate and the properties of the Sinkhorn algorithm, we first need to study the behavior of both the optimal potentials and the potentials generated by the Sinkhorn iteration. The potentials are determined only up to an additive constant, whereas their oscillations are intrinsic and independent of the chosen normalization. The following lemma shows that the oscillations of both the optimal potentials and the Sinkhorn iterates are controlled by the oscillation of the cost function.

For a bounded function $u$ defined on a set $A$, we write
\[
\mathrm{osc}_A(u) := \sup_{a, a' \in A} |u(a) - u(a')| = \sup_{a \in A} u(a) - \inf_{a \in A} u(a).
\]

A proof of the oscillation bound for the optimal potentials is given in \cite[Lemma 4.11]{nutz2022eot}. Here we use the same argument and apply it to the potentials generated by the Sinkhorn iteration.

\begin{lemma}
\label{lem:bound_potential}
Assume that the cost function \(c\) is bounded, with
\[
\|c\|_{L^\infty(E\times E)} \le M.
\]
Let \((\varphi,\psi)\) be optimal Schrödinger potentials, and let
\((\varphi^n,\psi^n)_{n\ge0}\) be the Sinkhorn iterates defined above. Then the following holds:

\begin{enumerate}
    \item \textbf{Optimal potentials.}
    \[
    \operatorname{osc}_E(\varphi)\le 2M,
    \qquad
    \operatorname{osc}_E(\psi)\le 2M.
    \]

    \item \textbf{Sinkhorn iterates.} For every \(n\ge0\),
    \[
    \operatorname{osc}_E(\varphi^{n+1})\le 2M,
    \qquad
    \operatorname{osc}_E(\psi^{n+1})\le 2M.
    \]
\end{enumerate}
\end{lemma}

\begin{proof}
Fix \(n\ge0\), and let \(z_1,z_2\in E\) be such that
\(\varphi^{n+1}(z_1)\ge \varphi^{n+1}(z_2)\). Then
\[
\begin{aligned}
\bigl|\varphi^{n+1}(z_1)-\varphi^{n+1}(z_2)\bigr|
&=
\log \int_E e^{-\psi^n(w)-c(z_2,w)}\,\mu(\mathrm{d}w)
-
\log \int_E e^{-\psi^n(w)-c(z_1,w)}\,\mu(\mathrm{d}w)
\\
&=
\log \int_E
e^{c(z_1,w)-c(z_2,w)-\psi^n(w)-c(z_1,w)}\,\mu(\mathrm{d}w)
-
\log \int_E e^{-\psi^n(w)-c(z_1,w)}\,\mu(\mathrm{d}w)
\\
&\le
\sup_{w\in E}\bigl(c(z_1,w)-c(z_2,w)\bigr)
\\
&\le
\sup_{w\in E}\bigl|c(z_1,w)-c(z_2,w)\bigr|.
\end{aligned}
\]
Taking the supremum over \(z_1,z_2\in E\), we obtain
\[
\operatorname{osc}_E(\varphi^{n+1})
\le
\sup_{w\in E}\operatorname{osc}_E\bigl(c(\cdot,w)\bigr)
\le 2M.
\]
The bound for \(\operatorname{osc}_E(\psi^{n+1})\) is obtained in the same way from the other Sinkhorn update.
The bounds for the optimal potentials \((\varphi,\psi)\) are also obtained analogously from the Schrödinger system.
\end{proof}

\begin{remark}[\(C^2\)-regularity of the potentials]
Assume moreover that \(c\) is \(C^2\) with respect to the continuous variables, and that all its first and second derivatives are bounded. By the logarithmic Sinkhorn update formula \eqref{eq:sinkhorn_update} and Lemma~\ref{lem:bound_potential}, the relevant exponential integrands and their first and second derivatives are dominated by integrable functions. Hence, by differentiation under the integral sign and the dominated convergence theorem, the Sinkhorn iterates \(\varphi^n,\psi^n\) are \(C^2\) with respect to the continuous variables. The same argument applied to the limiting Schrödinger system also gives the \(C^2\)-regularity of the optimal potentials \(\varphi,\psi\).
\end{remark}

We follow the semiconcavity approach of Chiarini, Conforti, Greco, and Tamanini \cite{ChiariniConfortiGrecoTamanini2024SemiconcStabConv}, who use semiconcavity estimates for Schrödinger potentials to establish stability of entropic plans and exponential convergence of Sinkhorn algorithm. We now turn to the study of stability properties of optimal couplings. The notion of semiconcavity of cost functions and potentials often turns out to be essential. We recall the definition of semiconcavity used below.A function \(f:\mathbb{R}^d\to\mathbb{R}\) is said to be \(\Lambda\)-semiconcave if, for every \(x,y\in\mathbb{R}^d\),
\[
f(y)-f(x)
\le
\langle \nabla f(x),y-x\rangle
+
\frac{\Lambda}{2}|y-x|^2.
\]
If \(f\) is \(C^2\), this condition is equivalent to the Hessian upper bound \(\nabla^2 f(x)\preceq \Lambda I\).

In the multi-regime state space \(E=\mathbb{R}^d\times\mathcal I\), a function \(F:E\to\mathbb{R}\) is said to be \(\Lambda\)-semiconcave if, for every regime \(i\in\mathcal I\), its restriction \(F_i(x):=F(x,i)\) is \(\Lambda\)-semiconcave on \(\mathbb{R}^d\). Equivalently, when each \(F_i\) is \(C^2\), this means that \(\nabla_x^2 F_i(x)\preceq \Lambda I\) for every \(x\in\mathbb{R}^d\) and \(i\in\mathcal I\).

\begin{lemma}
\label{lem:conditionlaw}
Assume that the cost function \(c\) is bounded, with
$
\|c\|_{L^\infty(E\times E)}\le M.
$
Let \(\pi\) be the optimal coupling between \(\rho\) and \(\mu\), and let
\((\varphi,\psi)\) be a pair of optimal Schrödinger potentials. Assume moreover
that, for every \(z_1\in E\) and every \(j\in\mathcal I\), the map
\[
y\longmapsto c\bigl(z_1,(y,j)\bigr)+\psi_j(y)
\]
is \(C^1\) and \(\Lambda\)-semiconcave on \(\mathbb R^d\). Then, for
\(\mu\)-a.e. \(z_2=(y,j)\) and \(z_2'=(y',j')\), one has
\[
H\bigl(\pi(\cdot\mid z_2)\mid \pi(\cdot\mid z_2')\bigr)
\le
\begin{cases}
\dfrac{\Lambda}{2}|y-y'|^2, & j=j',\\[0.6em]
4M, & j\neq j'.
\end{cases}
\]
\end{lemma}

\begin{proof}
By \eqref{eq:optimal_coupling_multiplicative} and the definitions
\(\varphi=-\log f\), \(\psi=-\log g\), and \(c=-\log r\), the optimal
coupling admits the logarithmic representation. Disintegrating
this representation with respect to its second marginal, we obtain that,
for \(\mu\)-a.e. \(z_2\in E\), the conditional law of the first variable
given the second one is
\[
\pi(\mathrm dz_1\mid z_2)
=
\exp\bigl(
-c(z_1,z_2)-\varphi(z_1)-\psi(z_2)
\bigr)\rho(\mathrm dz_1).
\]
Hence, for \(\mu\)-a.e. \(z_2,z_2'\in E\),
\[
\log
\frac{\mathrm d\pi(\cdot\mid z_2)}
{\mathrm d\pi(\cdot\mid z_2')}(z_1)
=
c(z_1,z_2')+\psi(z_2')
-
c(z_1,z_2)-\psi(z_2).
\]
Therefore,
\[
\begin{aligned}
H\bigl(\pi(\cdot\mid z_2)\mid \pi(\cdot\mid z_2')\bigr)
&=
\int_E
\bigl[
c(z_1,z_2')+\psi(z_2')
-
c(z_1,z_2)-\psi(z_2)
\bigr]
\pi(\mathrm dz_1\mid z_2).
\end{aligned}
\]
First suppose that \(j\neq j'\). Since \(|c|\le M\), and since
\(\operatorname{osc}_E(\psi)\le 2M\) by Lemma~\ref{lem:bound_potential}, we have
\[
c(z_1,z_2')+\psi(z_2')
-
c(z_1,z_2)-\psi(z_2)
\le 4M.
\]
Thus,
\[
H\bigl(\pi(\cdot\mid z_2)\mid \pi(\cdot\mid z_2')\bigr)
\le 4M.
\]
Now suppose that \(j=j'\). Write
$z_2=(y,j),  z_2'=(y',j)$.
For each \(z_1\in E\), define
\[
\ell_{z_1,j}(u)
:=
c\bigl(z_1,(u,j)\bigr)+\psi_j(u),
\qquad u\in\mathbb R^d.
\]
Then
\[
\begin{aligned}
H\bigl(\pi(\cdot\mid (y,j))\mid \pi(\cdot\mid (y',j))\bigr)
&=
\int_E
\bigl[
\ell_{z_1,j}(y')-\ell_{z_1,j}(y)
\bigr]
\pi(\mathrm dz_1\mid (y,j)).
\end{aligned}
\]
By \(\Lambda\)-semiconcavity of \(\ell_{z_1,j}\),
\[
\ell_{z_1,j}(y')-\ell_{z_1,j}(y)
\le
\left\langle
\nabla \ell_{z_1,j}(y),y'-y
\right\rangle
+
\frac{\Lambda}{2}|y'-y|^2.
\]
Integrating with respect to \(\pi(\mathrm dz_1\mid (y,j))\), we obtain
\[
\begin{aligned}
H\bigl(\pi(\cdot\mid (y,j))\mid \pi(\cdot\mid (y',j))\bigr)
&\le
\left\langle
\int_E
\nabla \ell_{z_1,j}(y)\,
\pi(\mathrm dz_1\mid (y,j)),
y'-y
\right\rangle
+
\frac{\Lambda}{2}|y'-y|^2.
\end{aligned}
\]
It remains to show that the linear term vanishes. The normalization identity
for the conditional law gives
\[
1
=
\int_E
\exp\bigl(
-c(z_1,(y,j))-\varphi(z_1)-\psi_j(y)
\bigr)
\rho(\mathrm dz_1).
\]
Differentiating this identity with respect to \(y\) yields
$\int_E \nabla \ell_{z_1,j}(y)\, \pi(\mathrm dz_1\mid (y,j)) = 0$.
Therefore,
\[
H\bigl(\pi(\cdot\mid (y,j))\mid \pi(\cdot\mid (y',j))\bigr)
\le
\frac{\Lambda}{2}|y-y'|^2.
\]
Combining the two cases concludes the proof.
\end{proof}

The following lemma verifies the semiconcavity condition required in the preceding lemma.
\begin{lemma}\label{lem:semiconcavity-psi}
Under Assumption~\ref{assum:map_c2}, for every $i,j\in\mathcal I$ with $\bar\rho(i)\bar\nu(j)>0$, every
$x\in \operatorname{supp}(\hat\rho_i)$, and every
$y\in \operatorname{supp}(\hat\nu_j)$, one has
\[
\nabla_y^2\bigl(c_{ij}(x,y)+\psi_j(y)\bigr)\preceq \Lambda I_d,
\qquad
\Lambda:=2M_2+M_1^2.
\]
In particular, for every fixed $i\in I$ and
$x\in \operatorname{supp}(\hat\rho_i)$, the map
\[
y\longmapsto c_{ij}(x,y)+\psi_j(y)
\]
is $\Lambda$-semiconcave on $\operatorname{supp}(\hat\nu_j)$.
\end{lemma}

\begin{proof}
Fix $j\in\mathcal I$ and $y\in \operatorname{supp}(\hat\nu_j)$. Denote by
$\pi\bigl(\mathrm d(x',i')\mid y,j\bigr)$
the conditional law of the first variable under $\pi$, given that the second variable
is equal to $(y,j)$. Note that by Assumptions \ref{assum:map_c2} and \ref{assum:measures}(2) we have that $ \psi_j $ is twice differentiable on $ \operatorname{supp}(\hat\nu_j) $.
Differentiating the logarithmic representation of $\psi_j$, we obtain
\[
\nabla_y^2\psi_j(y)
=
-
\int
\nabla_{yy}^2 c_{i'j}(x',y)\,
\pi\bigl(\mathrm d(x',i')\mid y,j\bigr)
+
\operatorname{Cov}_{(X',I')\sim\pi(\cdot\mid y,j)}
\!\bigl(\nabla_y c_{I'j}(X',y)\bigr).
\]
Therefore,
\[
\begin{aligned}
\nabla_y^2\bigl(c_{ij}(x,y)+\psi_j(y)\bigr)
&=
\nabla_{yy}^2 c_{ij}(x,y)
-
\int
\nabla_{yy}^2 c_{i'j}(x',y)\,
\pi\bigl(\mathrm d(x',i')\mid y,j\bigr) \\
&\quad+
\operatorname{Cov}_{(X',I')\sim\pi(\cdot\mid y,j)}
\!\bigl(\nabla_y c_{I'j}(X',y)\bigr).
\end{aligned}
\]
By the definition of $M_2$,
\[
\nabla_{yy}^2 c_{ij}(x,y)\preceq M_2 I_d,
\qquad
-
\int
\nabla_{yy}^2 c_{i'j}(x',y)\,
\pi\bigl(\mathrm d(x',i')\mid y,j\bigr)
\preceq M_2 I_d.
\]
Moreover, by the definition of $M_1$,
$\bigl\|\nabla_y c_{i'j}(x',y)\bigr\|\le M_1,$
hence
\[
\operatorname{Cov}_{(X',I')\sim\pi(\cdot\mid y,j)}
\!\bigl(\nabla_y c_{I'j}(X',y)\bigr)
\preceq
M_1^2 I_d.
\]
Combining the above estimates yields
\[
\nabla_y^2\bigl(c_{ij}(x,y)+\psi_j(y)\bigr)
\preceq
(2M_2+M_1^2)I_d.
\]
This proves the claim.
\end{proof}

To better compare probability distributions on the multi-regime space \(\mathbb{R}^d \times \mathcal{I}\), we introduce a new cost function adapted to its product structure. This cost compares points within the same regime through the usual squared Euclidean distance, while assigning a fixed unit cost to jumps between different regimes. We do not regard this function as a genuine metric, since it does not necessarily satisfy the triangle inequality. Nevertheless, it induces a natural Wasserstein-type transport cost, which will be useful in the stability estimates below.

Define the cost \(\omega\) on \((\mathbb{R}^d \times \mathcal I)^2\) by
\[
\omega\bigl((x,i),(y,j)\bigr) := |x-y|^2 \mathbf{1}_{\{i=j\}} + \mathbf{1}_{\{i\neq j\}}.
\]

Denote by \(W_\omega\) the optimal transport cost associated with the cost function \(\omega\), namely,

\[
W_\omega^2(\mu,\nu) := \inf_{\pi\in\Pi(\mu,\nu)} \int_{(\mathbb{R}^d\times \mathcal I)^2} \omega(z,z')\,\pi(\mathrm{d}z,\mathrm{d}z').
\]

For the proof of the stability estimate, we shall use the logarithmic Schrödinger potentials associated with the endpoint couplings introduced in Subsection~\ref{subsec:notation}. Fix $\rho \in \mathcal{P}(E)$, and let $\eta \in \mathcal{P}(E)$. We denote by $(\varphi^\eta,\psi^\eta)$ a pair of logarithmic Schrödinger potentials for $\pi^{\rho,\eta}$, so that
\[
\pi^{\rho,\eta}(\mathrm{d}z_1,\mathrm{d}z_2)
=
\exp
\left(
-c(z_1,z_2)
-
\varphi^\eta(z_1)
-
\psi^\eta(z_2)
\right)
\rho(\mathrm{d}z_1)\eta(\mathrm{d}z_2).
\]

\begin{lemma}[Stability]
\label{thm:stability_first}
Under the assumptions and notations of Theorem~\ref{thm:stable}, one has
\[
\mathcal H\bigl(Q^{\rho,\mu}\mid Q^{\rho,\nu}\bigr)
=
H\bigl(\pi^{\rho,\mu}\mid \pi^{\rho,\nu}\bigr)
\le
H(\mu\mid \nu)
+
B\,W_\omega^2(\mu,\nu).
\]
\end{lemma}

\begin{proof}
If $H(\mu\mid \nu)=+\infty$, then the claimed estimate is immediate. Hence,
we may assume that
\[
H(\mu\mid \nu)<+\infty,
\]
and therefore $\mu\ll\nu$.

Denote $ \varphi^{\mu} $, $ \psi^{\mu} $ the logarithmic Schr\"odinger potentials corresponding to endpoints marginals $ \rho, \mu $ and  $ \varphi^{\nu} $, $ \psi^{\nu} $ the logarithmic Schr\"odinger potentials corresponding to endpoints marginals $ \rho, \nu $.

By Lemma~\ref{lem:semiconcavity-psi}, for every $z_1=(x,i)\in \operatorname{supp}(\hat\rho_i)$ and every
$j\in\mathcal I$ with $\bar\nu(j)>0$, the map
\[
y \longmapsto c_{ij}(x,y)+\psi_j^\nu(y)
\]
is $\Lambda$-semiconcave on $\operatorname{supp}(\hat\nu_j)$ with
$\Lambda=2M_2+M_1^2$.
Since $|c|\le M_0$ on the relevant supports, Lemma~\ref{lem:conditionlaw}
applied to the optimal coupling $\pi^{\rho,\nu}$ yields that, for every
$b=(y,i), b'=(z,j)\in \mathbb R^d\times\mathcal I$,
\[
H\bigl(\pi^{\rho,\nu}(\cdot\mid b)\mid \pi^{\rho,\nu}(\cdot\mid b')\bigr)
\le
\max\left\{ \frac{\Lambda}{2},\,4M_0 \right\}\omega(b,b')
=
\max\left\{ M_2+\frac{M_1^2}{2},\,4M_0 \right\}\omega(b,b').
\]

Now let $\tau\in\Pi(\mu,\nu)$ be arbitrary, and disintegrate it with respect
to its first marginal:
\[
\tau(dz_2,dz)=\mu(dz_2)\,\tau(dz\mid z_2).
\]
Define a coupling $\tilde\pi\in\Pi(\rho,\mu)$ by
\[
\tilde\pi(dz_1,dz_2)
:=
\mu(dz_2)
\int_{\mathbb R^d\times\mathcal I}
\pi^{\rho,\nu}(dz_1\mid z)\,\tau(dz\mid z_2).
\]
Its first marginal is $\rho$ and its second marginal is $\mu$.

For
$\eta\in\{\mu,\nu\}$,
\[
\pi^{\rho,\eta}(dz_1,dz_2)
=
\exp\bigl(
-c(z_1,z_2)-\varphi^\eta(z_1)-\psi^\eta(z_2)
\bigr)
\,\rho(dz_1)\eta(dz_2).
\]
Consequently,
\[
\frac{d\pi^{\rho,\mu}}{d\pi^{\rho,\nu}}(z_1,z_2)
=
\exp\bigl(f(z_1)+g(z_2)\bigr),
\]
where
\[
f(z_1)
:=
\varphi^\nu(z_1)-\varphi^\mu(z_1)
\]
and
\[
g(z_2)
:=
\psi^\nu(z_2)-\psi^\mu(z_2)
+
\log\frac{d\mu}{d\nu}(z_2).
\]
Therefore,
by the EOT structure theorem,
\cite[Theorem~4.2(b)]{nutz2022eot}, applied with reference measure
$\pi^{\rho,\nu}$, the measure $\pi^{\rho,\mu}$ is the unique minimizer of
\[
\pi\longmapsto H\bigl(\pi\mid\pi^{\rho,\nu}\bigr)
\qquad\text{over }\Pi(\rho,\mu).
\]
In particular,
\[
H\bigl(\pi^{\rho,\mu}\mid\pi^{\rho,\nu}\bigr)
\le
H\bigl(\tilde\pi\mid\pi^{\rho,\nu}\bigr).
\]

We now estimate the right-hand side. Since the second marginal of
$\tilde\pi$ is $\mu$ and the second marginal of $\pi^{\rho,\nu}$ is $\nu$,
the chain rule for relative entropy gives
\[
H\bigl(\tilde\pi\mid\pi^{\rho,\nu}\bigr)
=
H(\mu\mid\nu)
+
\int
H\bigl(
\tilde\pi(\cdot\mid z_2)
\mid
\pi^{\rho,\nu}(\cdot\mid z_2)
\bigr)
\,\mu(dz_2).
\]
By construction,
\[
\tilde\pi(\cdot\mid z_2)
=
\int
\pi^{\rho,\nu}(\cdot\mid z)\,
\tau(dz\mid z_2).
\]
Using the convexity of relative entropy in its first argument, we obtain
\[
H\bigl(
\tilde\pi(\cdot\mid z_2)
\mid
\pi^{\rho,\nu}(\cdot\mid z_2)
\bigr)
\le
\int
H\bigl(
\pi^{\rho,\nu}(\cdot\mid z)
\mid
\pi^{\rho,\nu}(\cdot\mid z_2)
\bigr)
\,\tau(dz\mid z_2).
\]
Hence,
\[
\begin{aligned}
H\bigl(\tilde\pi\mid\pi^{\rho,\nu}\bigr)
&\le
H(\mu\mid\nu) \\
&\quad+
\int\!\!\int
H\bigl(
\pi^{\rho,\nu}(\cdot\mid z)
\mid
\pi^{\rho,\nu}(\cdot\mid z_2)
\bigr)
\,\tau(dz\mid z_2)\,\mu(dz_2).
\end{aligned}
\]
Applying the previous conditional entropy bound yields
\[
\begin{aligned}
H\bigl(\tilde\pi\mid\pi^{\rho,\nu}\bigr)
&\le
H(\mu\mid\nu) \\
&\quad+
\max\left\{
M_2+\frac{M_1^2}{2},\,4M_0
\right\}
\int
\omega(z,z_2)\,\tau(dz\mid z_2)\,\mu(dz_2).
\end{aligned}
\]
That is,
\[
H\bigl(\tilde\pi\mid\pi^{\rho,\nu}\bigr)
\le
H(\mu\mid\nu)
+
\max\left\{
M_2+\frac{M_1^2}{2},\,4M_0
\right\}
\int
\omega(z,z_2)\,\tau(dz_2,dz).
\]
Since $\tau\in\Pi(\mu,\nu)$ was arbitrary, taking the infimum over all such
couplings gives
\[
H\bigl(\pi^{\rho,\mu}\mid\pi^{\rho,\nu}\bigr)
\le
H(\mu\mid\nu)
+
\max\left\{
M_2+\frac{M_1^2}{2},\,4M_0
\right\}
W_\omega^2(\mu,\nu).
\]

Finally, by optimality \eqref{eq:optimal_bridge_conditionals}, both
$Q^{\rho,\mu}$ and $Q^{\rho,\nu}$ have the same conditional bridge law
$P(\cdot\mid Z_0,Z_T)$ given the endpoints. Therefore, the chain rule for
relative entropy on path space yields
\[
\begin{aligned}
\mathcal H\bigl(Q^{\rho,\mu}\mid Q^{\rho,\nu}\bigr)
&=
H\bigl(\pi^{\rho,\mu}\mid\pi^{\rho,\nu}\bigr) \\
&\quad+
\int
\mathcal H\bigl(
P(\cdot\mid z_0,z_T)
\mid
P(\cdot\mid z_0,z_T)
\bigr)
\,\pi^{\rho,\mu}(dz_0,dz_T).
\end{aligned}
\]
The second term is zero. Thus,
\[
\mathcal H\bigl(Q^{\rho,\mu}\mid Q^{\rho,\nu}\bigr)
=
H\bigl(\pi^{\rho,\mu}\mid\pi^{\rho,\nu}\bigr).
\]
Combining this identity with the endpoint estimate above concludes the proof.
\end{proof}

In the previous theorem, we obtained a preliminary stability estimate in terms of the transportation cost \(W_\omega\). However, unlike the classical Euclidean setting, the multi-regime space \(\mathbb{R}^d \times \mathcal{I}\) does not come with a canonical metric, and therefore the associated Wasserstein-type distance is not entirely intrinsic. The next theorem shows that, under componentwise \(T_2\) assumptions on the conditional laws, this transportation cost can be controlled by two more natural quantities, the relative entropy and the total variation distance. This provides a more convenient form of the stability estimate for the subsequent convergence analysis.

\begin{lemma}
\label{T2_discret}
Let \(\mu \in \mathcal P(\mathbb R^d \times \mathcal I)\). Assume that
\(\mu\) satisfies the Talagrand \(T_2\) inequality
\eqref{eq:t2_regimewise} with constant \(C_\mu>0\). Then, for every
\(\nu\in\mathcal{P}(\mathbb{R}^d\times\mathcal{I})\) such that
\(H(\nu\mid\mu)<\infty\), one has
\[
\begin{aligned}
W_\omega^2(\mu,\nu)
&\leq
\lVert \bar{\mu}-\bar{\nu}\rVert_{\mathrm{TV}}
+
C_\mu
\sum_{i:\,\bar{\mu}(i)>0,\ \bar{\nu}(i)>0}
\bigl(\bar{\nu}(i)\wedge\bar{\mu}(i)\bigr)
H\bigl(\hat{\nu}_i\mid\hat{\mu}_i\bigr)
\\
&\leq
\lVert \bar{\mu}-\bar{\nu}\rVert_{\mathrm{TV}}
+
C_\mu H(\nu\mid\mu).
\end{aligned}
\]
\end{lemma}

\begin{proof}
Recall that
\[
W_\omega^2(\mu,\nu)
=
\inf_{\pi\in\Pi(\mu,\nu)}
\int \omega(z,z')\,\pi(\mathrm{d}z,\mathrm{d}z').
\]
It is therefore enough to construct a coupling between \(\mu\) and \(\nu\) with the required cost. We split this coupling into two parts: first, we transport as much mass as possible without changing the regime index; then, we couple the remaining mass, which necessarily has to move between different regimes.

For each \(i\in\mathcal{I}\), set \(m_i:=\bar{\mu}(i)\wedge\bar{\nu}(i)\). This is the largest amount of mass that can be coupled while remaining in regime \(i\).
For every \(i\in\mathcal{I}\) such that \(m_i>0\), let
\[
\pi_i\in\Pi(\hat{\mu}_i,\hat{\nu}_i)
\]
be an optimal coupling for \(W_2^2(\hat{\mu}_i,\hat{\nu}_i)\). Thus, \(\pi_i\) is initially a probability measure on \(\mathbb{R}^d\times\mathbb{R}^d\). Via the embedding
\[
(x,y)\longmapsto \bigl((x,i),(y,i)\bigr),
\]
we identify \(\pi_i\) with a measure on \((\mathbb{R}^d\times\mathcal{I})^2\), supported on
$
(\mathbb{R}^d\times\{i\})
\times
(\mathbb{R}^d\times\{i\}).
$
Its first and second marginals are respectively \(\hat{\mu}_i\otimes\delta_i\) and \(\hat{\nu}_i\otimes\delta_i\), and
\[
\int \omega(z,z')\,\pi_i(\mathrm{d}z,\mathrm{d}z')
=
W_2^2(\hat{\mu}_i,\hat{\nu}_i).
\]
We define the part of the coupling which preserves the regime index by
\[
\pi^{\mathrm{same}}
:=
\sum_{i:\,m_i>0}m_i\pi_i.
\]
Its marginals are
\[
\mu^{\mathrm{same}}
:=
\sum_{i:\,m_i>0}
m_i\,\hat{\mu}_i\otimes\delta_i,
\qquad
\nu^{\mathrm{same}}
:=
\sum_{i:\,m_i>0}
m_i\,\hat{\nu}_i\otimes\delta_i.
\]

It remains to couple the residual measures
\[
\mu^{\mathrm{rem}}
:=
\mu-\mu^{\mathrm{same}},
\qquad
\nu^{\mathrm{rem}}
:=
\nu-\nu^{\mathrm{same}}.
\]
Their common total mass is
\[
\Delta
:=
1-\sum_{i\in\mathcal{I}}m_i
=
\lVert \bar{\mu}-\bar{\nu}\rVert_{\mathrm{TV}}.
\]
Moreover, for every \(i\in\mathcal{I}\), the residual masses in regime \(i\) are \(\bar{\mu}(i)-m_i\) and \(\bar{\nu}(i)-m_i\), respectively. Since \(m_i=\bar{\mu}(i)\wedge\bar{\nu}(i)\), at least one of these two quantities is zero. Hence, the two residual measures never charge the same regime, so any coupling of them is supported on pairs with distinct regime indices.
If \(\Delta>0\), define
\[
\pi^{\mathrm{rem}}
:=
\frac{1}{\Delta}\,
\mu^{\mathrm{rem}}\otimes\nu^{\mathrm{rem}};
\]
otherwise, set \(\pi^{\mathrm{rem}}:=0\). Then \(\pi^{\mathrm{rem}}\) is a coupling of \(\mu^{\mathrm{rem}}\) and \(\nu^{\mathrm{rem}}\), supported on pairs with distinct regime indices. Therefore,
\[
\int \omega(z,z')\,\pi^{\mathrm{rem}}(\mathrm{d}z,\mathrm{d}z')
=
\Delta.
\]

Finally, let \(\pi:=\pi^{\mathrm{same}}+\pi^{\mathrm{rem}}\). By construction, \(\pi\in\Pi(\mu,\nu)\). Hence,
\[
\begin{aligned}
W_\omega^2(\mu,\nu)
&\leq
\int \omega(z,z')\,\pi(\mathrm{d}z,\mathrm{d}z')
\\
&=
\sum_{i:\,m_i>0}
m_i W_2^2(\hat{\mu}_i,\hat{\nu}_i)
+
\lVert \bar{\mu}-\bar{\nu}\rVert_{\mathrm{TV}}
\\
&\leq
\lVert \bar{\mu}-\bar{\nu}\rVert_{\mathrm{TV}}
+
C_\mu
\sum_{i:\,m_i>0}
m_i H(\hat{\nu}_i\mid\hat{\mu}_i).
\end{aligned}
\]
Since \(H(\nu\mid\mu)<\infty\), one has \(\bar{\nu}\ll\bar{\mu}\). The entropy chain rule gives
\[
H(\nu\mid\mu)
=
H(\bar{\nu}\mid\bar{\mu})
+
\sum_{i:\,\bar{\nu}(i)>0}
\bar{\nu}(i)H(\hat{\nu}_i\mid\hat{\mu}_i).
\]
Since \(m_i\leq\bar{\nu}(i)\) and relative entropy is nonnegative,
\[
\sum_{i:\,m_i>0}
m_i H(\hat{\nu}_i\mid\hat{\mu}_i)
\leq
H(\nu\mid\mu).
\]
Combining the preceding estimates concludes the proof.
\end{proof}

\begin{proof}[Proof of Theorem~\ref{thm:stable}]
By Lemma~\ref{thm:stability_first}, with
$B=\max\left\{M_2+\frac{M_1^2}{2},\,4M_0\right\}$
we have
\[
\mathcal H(Q^{\rho,\mu}\mid Q^{\rho,\nu})
=
H(\pi^{\rho,\mu}\mid \pi^{\rho,\nu})
\le
H(\mu\mid \nu)+B W_\omega^2(\mu,\nu).
\]
Moreover, Lemma~\ref{T2_discret} yields
\[
W_\omega^2(\mu,\nu)
\leq
C_\mu H(\nu\mid\mu)
+
\lVert\bar{\mu}-\bar{\nu}\rVert_{\mathrm{TV}}.
\]
Combining the two estimates yields
\[
\mathcal H(Q^{\rho,\mu}\mid Q^{\rho,\nu})
=
H(\pi^{\rho,\mu}\mid \pi^{\rho,\nu})
\le
H(\mu\mid \nu)
+
BC_\mu H(\nu\mid \mu)
+
B\|\bar\mu-\bar\nu\|_{\mathrm{TV}}.
\]
Finally, since \(\|\bar\mu-\bar\nu\|_{\mathrm{TV}}\le \|\mu-\nu\|_{\mathrm{TV}}\), the last bound in the statement follows.
\end{proof}

For later convenience, we recall the following result of
\cite[Theorem~1.2]{eckstein2025hilbert}, using notation adapted to our setting.
It provides the exponential convergence of the Sinkhorn iterates in total
variation. Its proof relies on the contraction properties of the Sinkhorn maps
in Hilbert's projective metric.

\begin{theorem}[Exponential convergence of Sinkhorn iterates in total variation, \cite{eckstein2025hilbert}]
\label{tv_cv}
Let $(E,d_E)$ be a Polish space, and let $\rho,\mu\in\mathcal P(E)$.
Consider the entropic optimal transport problem on $E\times E$ with cost
function $c:E\times E\to\mathbb R$. Assume that there exist $p>0$ and
$\delta>0$ such that
\[
\lim_{r\to\infty}
\frac{\sup_{x,y\in B_r}c(x,y)}{r^p}
=0
\]
and
\[
\lim_{r\to\infty}
\frac{
\max\{
\rho(E\setminus B_r),\,
\mu(E\setminus B_r)
\}
}{
\exp(-r^{p+\delta})
}
=0,
\]
where $B_r$ denotes the ball of radius $r$ in $E$ with respect to $d_E$.
Let $(\pi^{n,n},\pi^{n+1,n})_{n\in\mathbb N}$ be the sequence of couplings
generated by Sinkhorn algorithm. Then there exist an optimal coupling
$\pi^{\rho,\mu}$, a constant $A>0$, and a constant $\kappa\in(0,1)$ such that, for all
$n\in\mathbb N$,
\[
\|\pi^{n,n}-\pi^{\rho,\mu}\|_{TV}\le A\kappa^n,
\qquad
\|\pi^{n+1,n}-\pi^{\rho,\mu}\|_{TV}\le A\kappa^n.
\]
\end{theorem}

\begin{corollary}
In our setting,
$
E=\mathbb R^d\times \mathcal I,
$
the cost function \(c\) is defined on
\((\mathbb R^d\times \mathcal I)^2\) and marginals \(\rho\) and \(\mu\) have
compact support.
Since only the restriction of \(c\) to
\(\operatorname{supp}(\rho)\times\operatorname{supp}(\mu)\) is relevant for
the EOT problem, one may, without loss of generality, assume that \(c\)
vanishes outside a compact subset of \((\mathbb R^d\times \mathcal I)^2\).
Then the assumptions of Theorem~\ref{tv_cv} are satisfied.
\end{corollary}

\begin{corollary}
\label{coro:tv}
The marginals defined in \eqref{eq:mu_definitions} converge exponentially
fast in total variation. More precisely, for all \(n\in\mathbb N\),
\[
\|\rho^{n,n}-\rho\|_{TV}\le A\kappa^n,
\qquad
\|\mu^{n+1,n}-\mu\|_{TV}\le A\kappa^n.
\]
\end{corollary}

To handle the recursion \eqref{eq:an_bound} obtained in the proof of Theorem~\ref{thm:exp}, we first record the following elementary estimate.
\begin{lemma}[Elementary recursion estimate]
\label{lem:elementary_recursion_estimate}
Let $(a_n)_{n\geq 0}$ be a sequence of nonnegative real numbers. Assume that there exist constants
$A,\alpha\in(0,1)$ and $D\geq 0$ such that, for every $n\geq 1$,
\[
a_n \leq A a_{n-1} + D\alpha^n .
\]
Then there exist constants $C>0$ and $\theta\in(0,1)$ such that
\[
a_n \leq C\theta^n,
\qquad n\geq 0.
\]
More precisely, one may take $\theta=A$ if $\alpha<A$, $\theta=\alpha$ if $\alpha>A$, and any
$\theta\in(A,1)$ if $\alpha=A$.
\end{lemma}

\begin{proof}
Dividing the recursion by $A^n$, we obtain
\[
\frac{a_n}{A^n}
\leq
\frac{a_{n-1}}{A^{n-1}}
+
D\left(\frac{\alpha}{A}\right)^n .
\]
Hence
\[
\frac{a_n}{A^n}
\leq
a_0
+
D\sum_{k=1}^n
\left(\frac{\alpha}{A}\right)^k
=
\begin{cases}
\displaystyle
a_0
+
D\frac{\alpha}{A}
\frac{1-\left(\frac{\alpha}{A}\right)^n}{1-\frac{\alpha}{A}},
&
\alpha<A,
\\[2ex]
\displaystyle
a_0+Dn,
&
\alpha=A,
\\[2ex]
\displaystyle
a_0
+
D\frac{\alpha}{A}
\frac{\left(\frac{\alpha}{A}\right)^n-1}{\frac{\alpha}{A}-1},
&
\alpha>A.
\end{cases}
\]
Consequently,
\[
a_n
\leq
\begin{cases}
\displaystyle
\left(
a_0
+
D\frac{\frac{\alpha}{A}}{1-\frac{\alpha}{A}}
\right)A^n
\leq
CA^n,
&
\alpha<A,
\\[2ex]
\displaystyle
(a_0+Dn)A^n
\leq
C(A+\varepsilon)^n,
&
\alpha=A,\quad 0<\varepsilon<1-A,
\\[2ex]
\displaystyle
a_0A^n
+
D\frac{\alpha(\alpha^n-A^n)}{\alpha-A}
\leq
C\alpha^n,
&
\alpha>A.
\end{cases}
\]
Thus $a_n$ decays exponentially in all cases.
\end{proof}

\begin{proof}[Proof of Theorem~\ref{thm:exp}]
Set
\[
a_n:=H(\pi^{\rho,\mu}\mid \pi^{n+1,n}).
\]
By the entropy identities for the Sinkhorn projections, see \cite[Proposition 6.5]{nutz2022eot},
\[
\begin{aligned}
a_n-a_{n-1}
&=H(\pi^{\rho,\mu}\mid \pi^{n+1,n})
  -H(\pi^{\rho,\mu}\mid \pi^{n,n-1})\\
&=-H(\rho\mid \rho^{n,n})
  -H(\mu\mid \mu^{n,n-1}).
\end{aligned}
\]
Equivalently,
\[
a_{n-1}-a_n
=
H(\rho\mid \rho^{n,n})
+
H(\mu\mid \mu^{n,n-1}).
\]
Since \(\pi^{n+1,n}=\pi^{\rho,\mu^{n+1,n}}\), applying
Theorem~\ref{thm:stable} with \(\nu=\mu^{n+1,n}\), we obtain
\[
a_n=H(\pi^{\rho,\mu}\mid \pi^{\rho,\mu^{n+1,n}})\le H(\mu\mid \mu^{n+1,n})+BC_\mu H(\mu^{n+1,n}\mid \mu)+B\|\mu^{n+1,n}-\mu\|_{\mathrm{TV}}.
\]
where \(C_\mu\) is the Talagrand constant in \eqref{eq:t2_regimewise} for \(\mu\).
By Sinkhorn monotonicity inequalities, see \cite[Proposition 6.10]{nutz2022eot},
\[
\begin{aligned}
H(\mu\mid \mu^{n+1,n})
&\le H(\mu\mid \mu^{n,n-1}),\\
H(\mu^{n+1,n}\mid \mu)
&\le H(\rho\mid \rho^{n,n}).
\end{aligned}
\]
Moreover, by Corollary~\ref{coro:tv}, there exist constants \(A_{\mathrm{TV}}>0\) and \(\kappa\in(0,1)\) such that
\[
\|\mu^{n+1,n}-\mu\|_{\mathrm{TV}}
\le A_{\mathrm{TV}}\kappa^n .
\]
Therefore, with
\[
C_{\mathrm{st}}:=\max\{1,BC_\mu\},
\]
we obtain
\[
\begin{aligned}
a_n
&\le
C_{\mathrm{st}}
\left(
H(\rho\mid \rho^{n,n})
+
H(\mu\mid \mu^{n,n-1})
\right)
+
BA_{\mathrm{TV}}\kappa^n\\
&=
C_{\mathrm{st}}(a_{n-1}-a_n)
+
BA_{\mathrm{TV}}\kappa^n .
\end{aligned}
\]
Hence
\begin{equation}\label{eq:an_bound}
a_n
\le
\frac{C_{\mathrm{st}}}{1+C_{\mathrm{st}}}a_{n-1}
+
\frac{BA_{\mathrm{TV}}}{1+C_{\mathrm{st}}}\kappa^n .
\end{equation}
Applying Lemma~\ref{lem:elementary_recursion_estimate} to
\eqref{eq:an_bound}, with
\(A=\frac{C_{\mathrm{st}}}{1+C_{\mathrm{st}}}\),
\(D=\frac{BA_{\mathrm{TV}}}{1+C_{\mathrm{st}}}\), and
\(\alpha=\kappa\), yields constants \(C>0\) and
\(\theta\in(0,1)\) such that
\[
H\bigl(\pi^{\rho,\mu} \mid \pi^{n+1,n}\bigr)
=
a_n
\leq
C\theta^n.
\]
Finally, by the entropy monotonicity along the second Sinkhorn projection,
see~\cite[Proposition~6.5]{nutz2022eot},
\[
H\bigl(\pi^{\rho,\mu} \mid \pi^{n+1,n+1}\bigr)
\leq
H\bigl(\pi^{\rho,\mu} \mid \pi^{n+1,n}\bigr).
\]
Therefore,
\[
H\bigl(\pi^{\rho,\mu} \mid \pi^{n+1,n+1}\bigr)
\leq
C\theta^n.
\]
\end{proof}

\section{Proofs of Theorems~\ref{thm:partial_observation}, ~\ref{thm:stability_partial} and~\ref{thm:po_exp_conv}}\label{sec:proof-partial}

The partially observed Schrödinger bridge problem in \eqref{eq:partial_schrodinger_bridge} is obtained by projecting the terminal endpoint \((X_T,I_T)\) onto its continuous component \(X_T\). The following lemma shows that this can be done in a general framework and under mild assumptions.

\begin{lemma}[EOT with partial observations]
\label{lem:partial_obeserve}
Let $X, Y, X', Y'$ be Borel spaces, and let $F : X \to X'$ and $G : Y \to Y'$ be measurable maps. Define
\[
T := F \otimes G : X \times Y \to X' \times Y', \qquad T(x,y) = (F(x),G(y)).
\]
Let $R \in \mathcal P(X \times Y)$ and $R' := T_{\#}R \in \mathcal P(X' \times Y')$. Given $\mu \in \mathcal P(X')$ and $\nu \in \mathcal P(Y')$, consider
\[
\inf \bigl\{ H(\pi \mid R) : \pi \in \mathcal P(X \times Y),\ T_{\#}\pi \in \Pi(\mu,\nu) \bigr\}
\qquad \text{and} \qquad
\inf_{\gamma \in \Pi(\mu,\nu)} H(\gamma \mid R').
\]

Assume that the right-hand problem admits a minimizer $\gamma^{*} \in \Pi(\mu,\nu)$ and that there exist measurable functions
$f : X' \to \mathbb R \cup \{-\infty\}$ and $g : Y' \to \mathbb R \cup \{-\infty\}$ such that
\[
\frac{\mathrm d \gamma^{*}}{\mathrm d R'}(x',y') = \exp\bigl(f(x') + g(y')\bigr).
\]

Then the left-hand problem admits a minimizer $\pi^{*} \in \mathcal P(X \times Y)$ satisfying
$T_{\#}\pi^{*} = \gamma^{*}$ and
$
H(\pi^{*} \mid R) = H(\gamma^{*} \mid R').
$
Moreover, $\pi^{*}$ admits the representation
\[
\frac{\mathrm d \pi^{*}}{\mathrm d R}(x,y) = \exp\bigl(f(F(x)) + g(G(y))\bigr).
\]
\end{lemma}

\begin{proof}
All Radon--Nikodym identities below are understood almost everywhere with
respect to the corresponding reference measure, and we use the convention
\(0\log 0=0\). Set
\[
h(x',y'):=\exp\bigl(f(x')+g(y')\bigr).
\]
By assumption,
\[
\frac{d\gamma^\ast}{dR'}=h.
\]
Define \(\pi^\ast\in\mathcal P(X\times Y)\) by
\[
\frac{d\pi^\ast}{dR}(x,y)
:=
h\bigl(T(x,y)\bigr)
=
\exp\bigl(f(F(x))+g(G(y))\bigr).
\]
Since \(R'=T_\#R\),
we have for every bounded measurable function \(\phi:X'\times Y'\to\mathbb{R}\),

\begin{align*}
    \int \phi(T) d\pi^\ast 
    = \int \phi(T) h(T) d R
    = \int \phi h d R' = \int \phi d \gamma^\ast.
\end{align*}

Therefore,
\[
T_\#\pi^\ast=\gamma^\ast\in\Pi(\mu,\nu),
\]
and thus \(\pi^\ast\) is feasible for the left-hand problem.

Now let \(\pi\in\mathcal P(X\times Y)\) be any feasible competitor for the
left-hand problem. Since \(T_\#\pi\in\Pi(\mu,\nu)\), the optimality of
\(\gamma^\ast\) and the data-processing inequality for relative entropy yield
\[
\begin{aligned}
H(\pi\mid R)
&\geq
H(T_\#\pi\mid T_\#R) \\
&=
H(T_\#\pi\mid R') \\
&\geq
H(\gamma^\ast\mid R').
\end{aligned}
\]
On the other hand, using again \(R'=T_\#R\), we obtain
\[
\begin{aligned}
H(\pi^\ast\mid R)
&=
\int_{X\times Y}
h\circ T\,\log(h\circ T)\,dR \\
&=
\int_{X'\times Y'}
h\log h\,dR' \\
&=
H(\gamma^\ast\mid R').
\end{aligned}
\]
Consequently, \(\pi^\ast\) attains the infimum in the left-hand problem and
\[
H(\pi^\ast\mid R)=H(\gamma^\ast\mid R').
\]
The asserted representation of \(\pi^\ast\) follows directly from its
definition.
\end{proof}

As in the fully observed case, the entropy decomposition~\eqref{eq:entropy_decomposition} reduces the dynamic problem to an optimization over the endpoint law. We then apply Lemma~\ref{lem:partial_obeserve} to obtain the factorization in terms of \(f\) and \(g\), from which the remaining assertions follow.

\begin{proof}
[Proof of Theorem~\ref{thm:partial_observation}]
We apply Lemma~\ref{lem:partial_obeserve}. Set
\[
X=Y=\mathbb R^d\times\mathcal I,
\qquad
X'=\mathbb R^d\times\mathcal I,
\qquad
Y'=\mathbb R^d.
\]
Let
\[
F=\mathrm{id}_{\mathbb R^d\times\mathcal I},
\qquad
G:\mathbb R^d\times\mathcal I\to \mathbb R^d,
\qquad
G(y,j)=y.
\]
Taking \(R=P_{0,T}\), its push-forward is
\[
R':=T_{\#}P_{0,T}=\overline P_{0,T}.
\]
Hence the projected static problem is
\[
\inf_{\gamma\in\Pi(\rho,\mu_{\mathrm p})}
H(\gamma\mid \overline P_{0,T}).
\]
By the classical entropic optimal transport result (\cite[Theorem 2.1]{nutz2022eot}), this problem admits a unique optimizer \(\gamma^{\rho,\mu_{\mathrm p}}\), and there exist nonnegative Schrödinger potentials \((f_i)_{i\in\mathcal I}\) and \(g\) such that
\[
\frac{\mathrm d\gamma^{\rho,\mu_{\mathrm p}}}{\mathrm dR'}((x,i),y)
=
f_i(x)g(y).
\]
Consequently, Lemma~\ref{lem:partial_obeserve} yields an optimizer \(\pi^{\rho,\mu_{\mathrm p}}\) for the partially observed endpoint problem \eqref{eq:eot-partial}, satisfying
\[
T_{\#}\pi^{\rho,\mu_{\mathrm p}}=\gamma^{\rho,\mu_{\mathrm p}}
\]
and
\[
\frac{\mathrm d\pi^{\rho,\mu_{\mathrm p}}}{\mathrm dP_{0,T}}\bigl((x,i),(y,j)\bigr)
=
f_i(x)g(y).
\tag{4.2}
\label{eq:fac-partial}
\]


\end{proof}

We first prove Theorem~\ref{thm:po_exp_conv}, since it follows directly by embedding the partially observed problem into a fully observed EOT problem. The proof of Theorem~\ref{thm:stability_partial} uses the same embedding but requires an additional argument to obtain the sharper estimate, and is therefore postponed.

\begin{proof}
[Proof of Theorem~\ref{thm:po_exp_conv}]
We prove the result by embedding the partially observed EOT problem into a fully observed EOT problem on an extended state space. Under this embedding, the partially observed problem becomes a particular case of the fully observed setting, and its Sinkhorn iterates coincide with the standard Sinkhorn iterates for the extended problem. The conclusion therefore follows directly from Theorem~\ref{thm:exp}.

Our current EOT problem is defined on the following spaces
\[
X := \mathbb{R}^d \times \mathcal{I},
\qquad
Y := \mathbb{R}^d,
\]
with the cost function
\[
\bar c : (\mathbb{R}^d \times \mathcal{I}) \times \mathbb{R}^d
\longrightarrow \mathbb{R}
\]
and the marginal distributions
\[
\rho \in \mathcal{P}(\mathbb{R}^d \times \mathcal{I})
\quad\text{and}\quad
\mu_{\mathrm p} \in \mathcal{P}(\mathbb{R}^d).
\]
To reduce it to the framework of the previous theorem, we define
\[
X_{\mathrm{ext}} := \mathbb{R}^d \times \mathcal{I},
\qquad
Y_{\mathrm{ext}} := \mathbb{R}^d \times \mathcal{I}.
\]

Next define the extended terminal marginal \(\mu_{\mathrm p}^{\mathrm{ext}}\in
\mathcal{P}(\mathbb{R}^d\times\mathcal{I})\) by
\[
\mu_{\mathrm p}^{\mathrm{ext}}(A\times\{1\}) := \mu_{\mathrm p}(A),
\qquad A\subset \mathbb{R}^d,
\]
and
\[
\mu_{\mathrm p}^{\mathrm{ext}}\bigl(\mathbb{R}^d\times\{2,\ldots,m\}\bigr) := 0.
\]

Define the extended cost
\[
c_{\mathrm{ext}} :
X_{\mathrm{ext}} \times Y_{\mathrm{ext}}
\longrightarrow \mathbb{R}
\]
by
\[
c_{\mathrm{ext}}\bigl((x,i),(y,j)\bigr)
:=
\begin{cases}
\bar c_i(x,y), & j=1,\\
0, & j\neq 1.
\end{cases}
\]
The value of \(c_{\mathrm{ext}}\) on the set \(\{j\neq 1\}\) is irrelevant for the extended EOT problem. Indeed, the extended terminal marginal \(\mu_{\mathrm p}^{\mathrm{ext}}\) is concentrated on \(\mathbb{R}^d\times\{1\}\), hence every admissible coupling is supported on
$X_{\mathrm{ext}}\times \bigl(\mathbb{R}^d\times\{1\}\bigr)$
Therefore only the values of \(c_{\mathrm{ext}}\) with \(j=1\) enter the objective, and the choice \(c_{\mathrm{ext}}=0\) for \(j\neq 1\) is only a convenient convention.

Let $\pi^{\mathrm{ext}}$ be an optimal coupling for the EOT problem with marginals $\rho$ and $\mu_{\mathrm p}^{\mathrm{ext}}$ and cost $c_{\mathrm{ext}}$. Since $\mu_{\mathrm p}^{\mathrm{ext}}$ is concentrated on $\mathbb R^d\times\{1\}$, $\pi^{\mathrm{ext}}$ is necessarily supported on $X_{\mathrm{ext}}\times(\mathbb R^d\times\{1\})$. It therefore induces a coupling $\pi\in\Pi(\rho,\mu_{\mathrm p})$ through
\[
\pi\bigl(d(x,i),dy\bigr)
=
\pi^{\mathrm{ext}}\bigl(d(x,i),d(y,1)\bigr).
\]
This coupling is optimal for the partially observed EOT problem. Conversely, every optimal coupling for the partially observed problem gives rise, through this embedding, to an optimal coupling for the extended problem.

Thus the partially observed EOT problem on \((\mathbb{R}^d\times\mathcal{I})\times\mathbb{R}^d\) is exactly equivalent to the standard EOT problem on \(X_{\mathrm{ext}}\times Y_{\mathrm{ext}}=(\mathbb{R}^d\times\mathcal{I})\times(\mathbb{R}^d\times\mathcal{I})\) with marginals \(\rho\) and \(\mu_{\mathrm p}^{\mathrm{ext}}\), and with cost \(c_{\mathrm{ext}}\).

In view of the above transformations, Theorem~\ref{thm:po_exp_conv} follows directly as a corollary of Theorem~\ref{thm:exp}.
\end{proof}

\begin{proof}[Proof of Theorem~\ref{thm:stability_partial}]
The proof continues to employ the state-space extension method introduced in Proof of Theorem~\ref{thm:po_exp_conv}. That is, the partial observation problem is treated as a full observation problem by extending the terminal marginals $\mu_{\mathrm p},\nu_{\mathrm p}\in\mathcal P(\mathbb R^d)$ to the product space $\mathbb R^d\times\mathcal I$, and concentrating all their mass on the first regime.

Denote the resulting extensions by \(\mu_{\mathrm p}^{\mathrm{ext}}\) and \(\nu_{\mathrm p}^{\mathrm{ext}}\). Their discrete marginals satisfy
\[
\overline{\mu_{\mathrm p}^{\mathrm{ext}}}
=
\overline{\nu_{\mathrm p}^{\mathrm{ext}}}
=
(1,0,\dots,0).
\]
Consequently,
\[
\left\|
\overline{\mu_{\mathrm p}^{\mathrm{ext}}}
-
\overline{\nu_{\mathrm p}^{\mathrm{ext}}}
\right\|_{\mathrm{TV}}
=0.
\]
This highlights a key distinction from the full observation case. In the full observation setting, Theorem~\ref{thm:stable} provides a stability estimate involving the Total Variation (TV) distance of the discrete marginals. However, in the current partial observation setting, this TV term vanishes automatically because both extended terminal marginals are concentrated solely on the first regime.
\end{proof}

\begin{remark}
Removing the TV term merely provides a more refined stability estimate and does not alter the ultimate conclusion of exponential convergence for the relative entropy. This is because, in the case of full observation, the TV term itself exhibits exponential decay—as established in Corollary~\ref{coro:tv}. However, it should be noted that while the property of exponential convergence remains unchanged, the specific convergence rate may be affected.
\end{remark}

\appendix
\section{Proof of Regularity of the Reference Process Density}
\label{app:c2}

In this appendix, we justify the regularity claim used in Example~\ref{ex:c2}. Under the constant-coefficient setting considered there, the reference process satisfies
\begin{equation}
\label{eq:ref_sde}
\mathrm d X_t = b_{I_t}\,\mathrm d t + \sigma_{I_t}\,\mathrm d W_t,
\qquad
X_0=x,\quad I_0=i,
\end{equation}
where \(\sigma_i>0\) for each \(i\in\mathcal I\). In this case, \((I_t)_{t\in[0,T]}\) is a continuous-time Markov chain on \(\mathcal I\) with strictly positive constant transition rates \((\lambda_{ij})_{i,j\in\mathcal I}\).

Recall that \(r_{ij}(0,x;T,y)\) denotes the transition density of the reference process, namely,
\[
\mathbb P\bigl(X_T\in \mathrm d y,\ I_T=j \mid X_0=x,\ I_0=i\bigr)
=
r_{ij}(0,x;T,y)\,\mathrm d y.
\]

We prove the following theorem.

\begin{theorem}
\label{thm:reference_density_c2}
For every \(T>0\) and every \(i,j\in\mathcal I\), the transition density
\[
r_{ij}(0,x;T,y)
\]
is strictly positive and of class \(C^2\) with respect to \((x,y)\) on \(\mathbb R^d\times\mathbb R^d\).
\end{theorem}
Note that the transition density $ r_{ij}( 0, x; T, y ) $ can be seen as the density of the law of the process \eqref{eq:ref_sde} at time $T$ and point $(y,j)$. A more precise meaning of this interpretation is given after Lemma \ref{app:lem_bound}.

We shall treat the reference SDE \eqref{eq:ref_sde} by a smoothing argument on the initial
distribution. Therefore, besides the transition density starting from a fixed
point $X_0=x,\ I_0=i$, we also consider the same dynamics with a more general initial
distribution $\mu_0$. More precisely, we consider
\begin{equation}
\begin{cases}
\mathrm dX_t=b_{I_t}\,\mathrm dt+\sigma_{I_t}\,\mathrm dW_t,\\
(X_0,I_0)\sim \mu_0,
\end{cases}
\label{eq:smoothed-sde}
\end{equation}
where $(X_0,I_0)$ is independent of $W$, and $I$ evolves with the same constant
transition rates $(\lambda_{ij})_{i,j\in\mathcal I}$.

For each $i\in\mathcal I$, we denote by $\mu_t^i$ the $i$-th regime component
of the law of $X_t$, namely
\begin{equation}
\label{eq:regime-marginal}
\mu_t^i(A)
:=
\mathbb P_{\mu_0}(X_t\in A,\ I_t=i),
\qquad A\in\mathcal B(\mathbb R^d).
\end{equation}

First, we prove that the solution to the above SDE \eqref{eq:smoothed-sde} admits a density with respect to the Lebesgue measure, and that this density has a bound. It is worth noting that the bound obtained here does not depend on the initial measure, nor does it require any smoothness of the initial measure.

This boundedness will be used in the PDE regularity estimate in Theorems~\ref{app:ref2} and~\ref{app:ref1}, in order to obtain improved regularity.

\begin{lemma}
\label{app:lem_bound}
Consider the process defined by the SDE \eqref{eq:smoothed-sde} with $ \mu_0 \in \mathcal{P}(\mathbb{R}^d) $. Set
\[
\underline a:=\min_{i\in\mathcal I}\sigma_i^2>0.
\]
Then, for every $t>0$ and every $i\in\mathcal I$, the measure $\mu_t^i$ is
absolutely continuous with respect to the Lebesgue measure. If $p_t^i$ denotes
its density, then
\[
0\le p_t^i(y)\le (2\pi \underline a t)^{-d/2},
\qquad y\in\mathbb R^d,\ i\in\mathcal I.
\]
\end{lemma}

\begin{proof}
Fix \(t>0\) and \(i\in\mathcal I\), and set
\(\mathcal G_t:=\sigma(X_0,(I_s)_{0\le s\le t})\). Conditionally on
\(\mathcal G_t\), one has
\[
X_t=X_0+\int_0^t b_{I_s}\,\mathrm ds+\int_0^t \sigma_{I_s}\,\mathrm dW_s .
\]
Hence \(X_t\) is conditionally Gaussian with mean
\(m_t:=X_0+\int_0^t b_{I_s}\,\mathrm ds\) and covariance matrix
\(v_t I_d\), where \(v_t:=\int_0^t \sigma_{I_s}^2\,\mathrm ds\). Since
\(v_t\ge \underline a t\), its conditional density \(q_t\) satisfies
\(0\le q_t(y)\le (2\pi \underline a t)^{-d/2}\) for all
\(y\in\mathbb R^d\).

Let \(A\in\mathcal B(\mathbb R^d)\). By the tower property,
\[
\mu_t^i(A)
=
\mathbb E_{\mu_0}\!\left[
\mathbf 1_{\{I_t=i\}}
\mathbb P_{\mu_0}(X_t\in A\mid \mathcal G_t)
\right]
=
\mathbb E_{\mu_0}\!\left[
\mathbf 1_{\{I_t=i\}}\int_A q_t(y)\,\mathrm dy
\right].
\]
By Fubini's theorem,
\[
\mu_t^i(A)
=
\int_A
\mathbb E_{\mu_0}\!\left[
\mathbf 1_{\{I_t=i\}}q_t(y)
\right]\mathrm dy .
\]
Therefore \(\mu_t^i\) is absolutely continuous with density
\(p_t^i(y):=\mathbb E_{\mu_0}[\mathbf 1_{\{I_t=i\}}q_t(y)]\). Moreover,
\[
0\le p_t^i(y)\le (2\pi \underline a t)^{-d/2},
\qquad y\in\mathbb R^d .
\]
This proves the claim.
\end{proof}

\begin{remark}
Note that the transition density in Theorem \ref{thm:reference_density_c2} is given by 
$ r_{ij}( 0, x; T, y ) = p_{ T }^{j }( y ) $ for $ \mu_0 = \delta_x\otimes\delta_i $ and $p$ being the density defined by Lemma \ref{app:lem_bound}.
\end{remark}

For measure-valued solutions, the evolution equation must first be interpreted in a suitable weak sense. Although the Fokker--Planck system of the regime-switching SDE holds distributionally, we shall use the distributional Duhamel formulation defined below. This formulation combines the distributional and Duhamel viewpoints, and is more convenient for the uniqueness and convergence argument used later in Lemma~\ref{app:lem_cv_mild}.

\begin{proposition}[Distributional Duhamel formula]
\label{app:pro_formula}
For each regime $i\in\mathcal I$, define the frozen-regime generator $L_i$ by
\[
L_i\phi(x)
:=
b_i\cdot \nabla \phi(x)
+
\frac{\sigma_i^2}{2}\Delta \phi(x),
\qquad \phi\in C_c^\infty(\mathbb R^d).
\]
Let $(P_t^i)_{t\ge0}$ be the Markov semigroup generated by $L_i$. 
For each \(i\in\mathcal I\), we use the component marginal \(\mu_t^i\) defined in \eqref{eq:regime-marginal} for $ \mu_0 \in \mathcal{P}(\mathbb{R}^d) $.

Then, for every $t\ge0$, every $\phi\in C_c^\infty(\mathbb R^d)$, and every $i\in\mathcal I$, one has
\begin{equation}
\label{eq:component-mild-form}
\begin{aligned}
\langle \mu_t^i,\phi\rangle
&=
\langle \mu_0^i,P_t^i\phi\rangle
+
\int_0^t
\left[
-
\sum_{j\neq i}
\lambda_{ij}
\langle \mu_s^i,P_{t-s}^i\phi\rangle
+
\sum_{k\neq i}
\lambda_{ki}
\langle \mu_s^k,P_{t-s}^i\phi\rangle
\right]
\mathrm ds .
\end{aligned}
\end{equation}
\end{proposition}

\begin{proof}
Fix \(t>0\), \(i\in\mathcal I\), and \(\phi\in C_c^\infty(\mathbb R^d)\). Set
\(u_i(s,x):=P_{t-s}^i\phi(x)\) for \(0\le s\le t\). Then
\(u_i\) solves the backward equation \(\partial_s u_i+L_i u_i=0\) with
terminal condition \(u_i(t,\cdot)=\phi\).

For \(k\neq \ell\), let
\(N_t^{k\ell}:=\sum_{0<s\le t}\mathbf 1_{\{I_{s-}=k,\ I_s=\ell\}}\)
be the counting process of jumps from \(k\) to \(\ell\). Since the
transition rates are constant,
\(\widetilde N_t^{k\ell}:=N_t^{k\ell}
-\int_0^t\lambda_{k\ell}\mathbf 1_{\{I_{s-}=k\}}\,\mathrm ds\)
is a martingale.

We apply Itô's formula with jumps to
\(Z_s^i:=\mathbf 1_{\{I_s=i\}}u_i(s,X_s)\). The backward equation cancels
the continuous diffusion part. Taking expectations, the martingale terms
vanish, and the compensators of the jump counting processes give
\[
\begin{aligned}
\mathbb E\!\left[\mathbf 1_{\{I_t=i\}}\phi(X_t)\right]
&=
\mathbb E\!\left[\mathbf 1_{\{I_0=i\}}P_t^i\phi(X_0)\right]  \\
&\quad+
\int_0^t
\left[
-\sum_{j\neq i}\lambda_{ij}
\mathbb E\!\left[
\mathbf 1_{\{I_s=i\}}P_{t-s}^i\phi(X_s)
\right]
+
\sum_{k\neq i}\lambda_{ki}
\mathbb E\!\left[
\mathbf 1_{\{I_s=k\}}P_{t-s}^i\phi(X_s)
\right]
\right]\mathrm ds .
\end{aligned}
\]
Finally, using
\(\langle\mu_s^k,\psi\rangle
=\mathbb E[\mathbf 1_{\{I_s=k\}}\psi(X_s)]\), this identity becomes
exactly the desired distributional Duhamel formula.
\end{proof}

We next prove uniqueness for the distributional Duhamel system. This will allow us to identify any limiting measure-valued evolution obtained from the smoothed approximations.

\begin{theorem}[Uniqueness for the distributional Duhamel system]
\label{app:thm_uni}
Let $T>0$. Let
$(\mu_t^i)_{t\in[0,T],\,i\in\mathcal I}$ and
$(\widetilde\mu_t^i)_{t\in[0,T],\,i\in\mathcal I}$
be two families of finite nonnegative Borel measures on $\mathbb R^d$.
Assume that both families satisfy the distributional Duhamel formula~\eqref{eq:component-mild-form} with the same initial data, namely
\[
\mu_0^i=\widetilde\mu_0^i,
\qquad i\in\mathcal I.
\]
Assume moreover that
\[
\sup_{t\in[0,T]}
\sum_{i\in\mathcal I}
\left(
\mu_t^i(\mathbb R^d)
+
\widetilde\mu_t^i(\mathbb R^d)
\right)
<\infty .
\]
Then, for every $t\in[0,T]$ and every $i\in\mathcal I$,
\[
\mu_t^i=\widetilde\mu_t^i .
\]
In particular, the distributional Duhamel system admits at most one solution
in this class.
\end{theorem}

\begin{proof}
Set \(\Delta_t^i:=\mu_t^i-\widetilde\mu_t^i\). Since the two families have the same initial data, subtracting the two Duhamel formulas gives, for every admissible test function \(\phi\),
\[
\langle \Delta_t^i,\phi\rangle
=
\int_0^t
\left[
-\sum_{j\neq i}\lambda_{ij}
\langle \Delta_s^i,P_{t-s}^i\phi\rangle
+
\sum_{k\neq i}\lambda_{ki}
\langle \Delta_s^k,P_{t-s}^i\phi\rangle
\right]\mathrm ds .
\]
Since \(P_t^i\) is a Markov semigroup, it is a contraction in the sup norm. Hence, for \(\|\phi\|_\infty\le 1\),
\[
|\langle \Delta_t^i,\phi\rangle|
\le
\int_0^t
\left[
\sum_{j\neq i}\lambda_{ij}\|\Delta_s^i\|_{\mathrm{TV}}
+
\sum_{k\neq i}\lambda_{ki}\|\Delta_s^k\|_{\mathrm{TV}}
\right]\mathrm ds .
\]
Taking the supremum over \(\|\phi\|_\infty\le 1\), we obtain
\[
\|\Delta_t^i\|_{\mathrm{TV}}
\le
\int_0^t
\left[
\sum_{j\neq i}\lambda_{ij}\|\Delta_s^i\|_{\mathrm{TV}}
+
\sum_{k\neq i}\lambda_{ki}\|\Delta_s^k\|_{\mathrm{TV}}
\right]\mathrm ds .
\]
Let
\[
D(t):=\sum_{i\in\mathcal I}\|\Delta_t^i\|_{\mathrm{TV}},
\qquad
\Lambda:=\max_{i\in\mathcal I}\sum_{j\neq i}\lambda_{ij}.
\]
Summing the previous inequality over \(i\in\mathcal I\) yields
\[
D(t)\le 2\Lambda\int_0^t D(s)\,\mathrm ds .
\]
By Gronwall's lemma, \(D(t)=0\) for every \(t\in[0,T]\). Therefore
\(\Delta_t^i=0\) for every \(t\in[0,T]\) and every \(i\in\mathcal I\), namely
\(\mu_t^i=\widetilde\mu_t^i\). This proves uniqueness.
\end{proof}

We next define the smoothing procedure precisely and record the PDE system satisfied by the corresponding densities.

\begin{definition}[Smoothed approximation of the initial Dirac mass]
\label{app:def_sm}
Fix an initial point $(x,i)\in \mathbb R^d\times\mathcal I$ and let
$\varepsilon>0$. We define the smoothed initial distribution
$\mu_0^\varepsilon\in\mathcal P(\mathbb R^d\times\mathcal I)$ by smoothing only
the continuous component and keeping the initial regime fixed. More precisely,
\[
\mu_0^\varepsilon(\mathrm dz,\{k\})
=
\mathbf 1_{\{k=i\}}
(2\pi\varepsilon)^{-d/2}
\exp\left(
-\frac{|z-x|^2}{2\varepsilon}
\right)
\mathrm dz,
\qquad k\in\mathcal I .
\]
Let $(X_t^\varepsilon,I_t^\varepsilon)_{t\ge0}$ be the solution of the same
regime-switching SDE with initial distribution $\mu_0^\varepsilon$, namely
\[
\begin{cases}
\mathrm dX_t^\varepsilon
=
b_{I_t^\varepsilon}\,\mathrm dt
+
\sigma_{I_t^\varepsilon}\,\mathrm dW_t,\\
(X_0^\varepsilon,I_0^\varepsilon)\sim \mu_0^\varepsilon,
\end{cases}
\]
where $(X_0^\varepsilon,I_0^\varepsilon)$ is independent of $W$, and
$I^\varepsilon$ evolves with the same constant transition rates
$(\lambda_{kl})_{k,l\in\mathcal I}$.

For each $k\in\mathcal I$, define the regime-wise marginal measure
\[
\mu_t^{\varepsilon,k}(A)
:=
\mathbb P(X_t^\varepsilon\in A,\ I_t^\varepsilon=k),
\qquad A\in\mathcal B(\mathbb R^d).
\]
By Lemma~\ref{app:lem_bound}, for every $t>0$, the measure $\mu_t^{\varepsilon,k}$ admits a
density $p_t^{\varepsilon,k}$ with respect to the Lebesgue measure, namely
\[
\mu_t^{\varepsilon,k}(\mathrm dz)
=
p_t^{\varepsilon,k}(z)\,\mathrm dz .
\]
Moreover,
\[
0\le p_t^{\varepsilon,k}(z)
\le
(2\pi\underline a t)^{-d/2},
\qquad t>0,\ z\in\mathbb R^d,\ k\in\mathcal I .
\]
The initial data are given by
\[
p_0^{\varepsilon,k}(z)
=
\mathbf 1_{\{k=i\}}
(2\pi\varepsilon)^{-d/2}
\exp\left(
-\frac{|z-x|^2}{2\varepsilon}
\right),
\qquad k\in\mathcal I .
\]
For $t>0$, the density vector
$p^\varepsilon=(p^{\varepsilon,k})_{k\in\mathcal I}$ satisfies the forward
system
\begin{equation}
\label{eq:regularised}
\partial_t p_t^{\varepsilon,k}
=
-b_k\cdot\nabla p_t^{\varepsilon,k}
+
\frac{\sigma_k^2}{2}\Delta p_t^{\varepsilon,k}
-
\sum_{l\neq k}\lambda_{kl}p_t^{\varepsilon,k}
+
\sum_{l\neq k}\lambda_{lk}p_t^{\varepsilon,l},
\qquad k\in\mathcal I .
\end{equation}
\end{definition}

Although \(p^\varepsilon=(p^{\varepsilon,k})_{k\in\mathcal I}\) solves a coupled forward system \eqref{eq:regularised}, we analyze it component by component. For each fixed \(k\in\mathcal I\), the \(k\)-th equation can be rewritten as the scalar parabolic equation
\[
\partial_t p^{\varepsilon,k}_t
+ b_k\cdot \nabla p^{\varepsilon,k}_t
-\frac{\sigma_k^2}{2}\Delta p^{\varepsilon,k}_t
= f^{\varepsilon,k}_t,
\]
where
\[
f^{\varepsilon,k}_t
:=
-\sum_{\ell\neq k}\lambda_{k\ell}p^{\varepsilon,k}_t
+\sum_{\ell\neq k}\lambda_{\ell k}p^{\varepsilon,\ell}_t .
\]
This reformulation allows us to bootstrap the regularity of \(p^\varepsilon\). Lemma~\ref{app:lem_bound} gives uniform \(L^\infty\) bounds on \(p^{\varepsilon,k}\), hence \(f^{\varepsilon,k}\) is uniformly bounded. Interior parabolic \(L^p\) estimates, together with the parabolic Sobolev--Hölder embedding, yield Hölder regularity of \(p^{\varepsilon,k}\). Hence \(f^{\varepsilon,k}\) is Hölder continuous as well. Applying the Schauder estimate to the scalar equation gives \(C^{1+\alpha/2,2+\alpha}_{\mathrm{loc}}\) regularity of \(p^{\varepsilon,k}\). Doing this for each \(k\in\mathcal I\) gives the desired regularity of \(p^\varepsilon\).

The following interior parabolic $L^p$ regularity estimate is adapted from ~\cite[Theorem~7.22]{lieberman1996second}.
\begin{theorem}
[Interior parabolic $L^p$ regularity estimate]
\label{app:ref2}
Assume that $u$ satisfies
\[
\partial_t u + b \cdot \nabla u - \frac{\sigma^2}{2}\Delta u = f
\qquad \text{in } (\tau-\delta,T+\delta)\times B_R,
\]
and that $u$ is smooth. Then
\[
\|\partial_t u\|_{L^p((\tau,T)\times B_{R/2})}
+
\|\nabla u\|_{L^p((\tau,T)\times B_{R/2})}
+
\|D_x^2 u\|_{L^p((\tau,T)\times B_{R/2})}
\le
C\Bigl(
\|f\|_{L^p((\tau-\delta,T+\delta)\times B_R)}
+
\|u\|_{L^p((\tau-\delta,T+\delta)\times B_R)}
\Bigr),
\]
where $b\in\mathbb{R}^d$ and $\sigma\neq 0$ are constants, and
$
C=C(p,\tau,T,R,\delta,b,\sigma).
$
\end{theorem}

We recall the following parabolic Sobolev--Hölder embedding, which follows from \cite[Chapter II, Lemma~3.3]{Ladyzhenskaya1968} by taking l=1 and r=s=0.

\begin{theorem}[Parabolic Sobolev--H\"older embedding]
\label{app:ref3}
Let \(n \ge 1\), \(1<q<\infty\), \(q>\frac{n+2}{2}\), and \(0<\alpha<2-\frac{n+2}{q}\). Then
\[
W^{2,1}_q\big([\tau,T]\times B_R\big)\hookrightarrow
C^{\alpha,\alpha/2}\big([\tau,T]\times B_R\big).
\]
More precisely, there exists a constant \(C=C(n,q,\alpha,R,\tau,T)>0\) such that, for every
\[
u\in W^{2,1}_q\big([\tau,T]\times B_R\big),
\]
one has
\[
\|u\|_{C^{\alpha,\alpha/2}\big([\tau,T]\times B_R\big)}
\le
C\,\|u\|_{W^{2,1}_q\big([\tau,T]\times B_R\big)}.
\]
\end{theorem}

The following interior parabolic Schauder estimate is adapted from \cite[Ch.~8]{Krylov1996}.
\begin{theorem}[Interior parabolic Schauder estimate]
\label{app:ref1}
Let $\alpha\in(0,1)$. Suppose $u$ satisfies
\[
\partial_t u+b\cdot \nabla u-\frac{\sigma^2}{2}\Delta u=f
\qquad \text{in }(0,T)\times B_R,
\]
where $b\in\mathbb R^n$ and $\sigma>0$ are constants. If
\[
f\in C^{\alpha/2,\alpha}((0,T)\times B_R),
\]
then
\[
\|u\|_{C^{\,1+\alpha/2, 2+\alpha}((0,T)\times B_{R/2})}
\le
C\Big(
\|f\|_{C^{\alpha/2,\alpha}((0,T)\times B_R)}
+\|u\|_{L^\infty((0,T)\times B_R)}
\Big),
\]
for some constant $C=C(n,\alpha,\sigma,b,R,T)$.
\end{theorem}

The above estimates ensure that, up to extracting a subsequence, \(p^\varepsilon\) admits a \(C^{1,2}\) limit.

\begin{lemma}
\label{app:lem_cv}
Let $0<\tau<T$, $R>0$, and $\alpha\in(0,1)$. Then there exists a constant
$C=C(\tau,T,R,\alpha)>0$, independent of $\varepsilon$, such that for every
$k\in\mathcal I$,
\[
\|p^{\varepsilon,k}\|_{C^{1+\alpha/2,\,2+\alpha}([\tau,T]\times B_R)}
\le C .
\]
Consequently, for every sequence $\varepsilon_n\downarrow0$, there exist a
subsequence, still denoted by $\varepsilon_n$, and functions
$(p^k)_{k\in\mathcal I}$ such that, for every $k\in\mathcal I$,
\[
p^{\varepsilon_n,k}\longrightarrow p^k
\quad\text{in } C^{1,2}_{\mathrm{loc}}\bigl((0,T]\times\mathbb R^d\bigr).
\]
In particular,
\[
p^k\in C^{1,2}\bigl((0,T]\times\mathbb R^d\bigr),
\qquad k\in\mathcal I .
\]
\end{lemma}

\begin{proof}
Fix $0<\tau<T$ and $R>0$. By Lemma~\ref{app:lem_bound}, for every $t>0$,
$z\in\mathbb R^d$, and $k\in\mathcal I$,
\[
0\le p_t^{\varepsilon,k}(z)
\le
(2\pi \underline a t)^{-d/2}.
\]
Hence the family $(p^{\varepsilon,k})_{\varepsilon>0,\,k\in\mathcal I}$ is
uniformly bounded on $[\tau/2,T]\times\mathbb R^d$.

Moreover, by Definition~\ref{app:def_sm}, the vector
$p^\varepsilon=(p^{\varepsilon,k})_{k\in\mathcal I}$ solves the linear
constant-coefficient forward system
\[
\partial_t p_t^{\varepsilon,k}
=
-b_k\cdot\nabla p_t^{\varepsilon,k}
+
\frac{\sigma_k^2}{2}\Delta p_t^{\varepsilon,k}
-
\sum_{l\neq k}\lambda_{kl}p_t^{\varepsilon,k}
+
\sum_{l\neq k}\lambda_{lk}p_t^{\varepsilon,l},
\qquad k\in\mathcal I .
\]
The coupling terms are therefore uniformly bounded on
$[\tau/2,T]\times\mathbb R^d$. 
Applying
Theorems~\ref{app:ref2}, \ref{app:ref3} and \ref{app:ref1} to this system
gives, for each \(k\in\mathcal I\), a constant \(C=C(\tau,T,R,\alpha)>0\),
independent of \(\varepsilon\), such that
\[
\|p^{\varepsilon,k}\|_{C^{1+\alpha/2,\,2+\alpha}([\tau,T]\times B_R)}
\le C .
\]
In particular, the functions \(p^{\varepsilon,k}\), together with their first
time derivatives and first and second spatial derivatives, are uniformly
bounded and equicontinuous on \([\tau,T]\times B_R\).

By the Arzelà--Ascoli theorem, a diagonal extraction yields a subsequence, denoted by \(\varepsilon_n\), and functions
\((p^k)_{k\in\mathcal I}\) such that, for every \(k\in\mathcal I\),
\[
p^{\varepsilon_n,k}\longrightarrow p^k
\quad\text{in } C_{\mathrm{loc}}^{1,2}((0,T]\times\mathbb R^d).
\]
Consequently,
\[
p^k\in C^{1,2}((0,T]\times\mathbb R^d),
\qquad k\in\mathcal I .
\]
\end{proof}

We have obtained a limit from the smoothed approximations. We now show that this limit still satisfies the distributional Duhamel formula.
\begin{lemma}
\label{app:lem_cv_mild}
Let $(p^i)_{i\in\mathcal I}$ be the limit obtained in the Lemma~\ref{app:lem_cv}, and
for each $t\in(0,T]$ define
\[
\mu_t^i(\mathrm dx)
:=
p_t^i(x)\,\mathrm dx,
\qquad i\in\mathcal I.
\]
Together with the prescribed initial measures $(\mu_0^i)_{i\in\mathcal I}$,
the family $(\mu_t^i)_{t\in[0,T],\,i\in\mathcal I}$ satisfies the distributional
Duhamel system. More precisely, for every $t\in[0,T]$, every
$\phi\in C_c^\infty(\mathbb R^d)$, and every $i\in\mathcal I$,
\[
\begin{aligned}
\langle \mu_t^i,\phi\rangle
=
\langle \mu_0^i,P_t^i\phi\rangle
+
\int_0^t
\left[
-
\sum_{j\neq i}
\lambda_{ij}
\langle \mu_s^i,P_{t-s}^i\phi\rangle
+
\sum_{k\neq i}
\lambda_{ki}
\langle \mu_s^k,P_{t-s}^i\phi\rangle
\right]
\mathrm ds.
\end{aligned}
\]
\end{lemma}

\begin{proof}
For every $\varepsilon>0$, the family
$(\mu_t^{\varepsilon,i})_{i\in\mathcal I}$ associated with the smoothed
approximation satisfies the distributional Duhamel formula. Hence, for every
$t\in[0,T]$, every $\phi\in C_c^\infty(\mathbb R^d)$, and every
$i\in\mathcal I$,
\[
\begin{aligned}
\langle \mu_t^{\varepsilon,i},\phi\rangle
=
\langle \mu_0^{\varepsilon,i},P_t^i\phi\rangle
+
\int_0^t
\left[
-
\sum_{j\neq i}
\lambda_{ij}
\langle \mu_s^{\varepsilon,i},P_{t-s}^i\phi\rangle
+
\sum_{k\neq i}
\lambda_{ki}
\langle \mu_s^{\varepsilon,k},P_{t-s}^i\phi\rangle
\right]
\mathrm ds.
\end{aligned}
\]
Letting $\varepsilon\downarrow0$, the initial measures
$\mu_0^{\varepsilon,i}$ converge weakly to $\mu_0^i$, while the previous lemma
gives
\[
p^{\varepsilon,i}\longrightarrow p^i
\quad\text{in } C^{1,2}_{\mathrm{loc}}\bigl((0,T]\times\mathbb R^d\bigr),
\qquad i\in\mathcal I .
\]
Since \(\phi\) is compactly supported and \(P_t^i\phi\) is bounded and smooth, the initial term and the first integral term pass to the limit directly. For the last integral term, the convergence \(p^{\varepsilon,k}\to p^k\), together with the uniform bounds obtained above, allows us to apply the dominated convergence theorem. We may therefore pass to the limit in the Duhamel formula, which yields the claimed identity.
\end{proof}

\begin{proof}[Proof of Theorem~\ref{thm:reference_density_c2}]
Fix \((x,i)\in \mathbb R^d\times\mathcal I\). Approximate
\(\mu_0 = \delta_x\otimes\delta_i\) by the smoothed initial laws $\mu_0^\varepsilon$ of Definition~\ref{app:def_sm}.
By Lemma~\ref{app:lem_bound}, the family of regime-wise densities
\(\{p^{\varepsilon,k}\}_{\varepsilon>0}\) is uniformly bounded.
Moreover, the regularity estimates in Theorems~
\ref{app:ref2}, \ref{app:ref3} and \ref{app:ref1}, applied successively, imply the corresponding local
equicontinuity in \(C^{1,2}\bigl((0,T]\times\mathbb R^d\bigr)\).
Therefore, by the Arzelà--Ascoli theorem, there
exists a subsequence, still denoted by \(p^{\varepsilon,k}\), and a limit
\(p^k\in C^{1,2}\bigl((0,T]\times\mathbb R^d\bigr)\), such that
\(p^{\varepsilon,k}\to p^k\) locally in
\(C^{1,2}\bigl((0,T]\times\mathbb R^d\bigr)\).

It remains to identify this limit. For every \(\varepsilon>0\), the measures
\(p_t^{\varepsilon,k}(y)\,\mathrm dy\) satisfy the distributional Duhamel
formula~\eqref{eq:component-mild-form} in Proposition~\ref{app:pro_formula}. Passing to the limit, the measures
\(p_t^k(y)\,\mathrm dy\) satisfy the same Duhamel system with initial data
\(\delta_x\otimes\delta_i\) by Lemma~\ref{app:lem_cv_mild}. The law of the reference process started from
\((x,i)\) also satisfies this system. Therefore, by the uniqueness result of
Theorem~\ref{app:thm_uni}, the two measure-valued solutions coincide. Consequently,
\[
r_{ij}(0,x;T,y)=p_T^j(y).
\]
Thus \(y\mapsto r_{ij}(0,x;T,y)\) is of class \(C^2\).

Since the coefficients are constant, the reference dynamics are spatially
translation invariant. Hence
\[
r_{ij}(0,x;T,y)=r_{ij}(0,0;T,y-x),
\]
which implies that \((x,y)\mapsto r_{ij}(0,x;T,y)\) is of class \(C^2\).

It remains to prove strict positivity. Set
\(\Lambda_k:=\sum_{\ell\neq k}\lambda_{k\ell}\). If \(i=j\), we consider the
event that no regime switch occurs on \([0,T]\). This event has probability
\(e^{-\Lambda_iT}>0\), and on this event
\[
X_T=x+b_iT+\sigma_iW_T.
\]
Hence the law of \(X_T\) contains a non-degenerate Gaussian contribution,
which is everywhere strictly positive. Therefore
\(r_{ii}(0,x;T,y)>0\) for every \(y\in\mathbb R^d\).

If \(i\neq j\), we consider the trajectories with exactly one regime switch,
from \(i\) to \(j\), occurring at some time \(s\in(0,T)\). The corresponding
contribution to the transition density is
\[
\int_0^T
e^{-\Lambda_i s}\lambda_{ij}e^{-\Lambda_j(T-s)}
g_{\sigma_i^2s+\sigma_j^2(T-s)}
\bigl(y-x-b_is-b_j(T-s)\bigr)\,\mathrm ds,
\]
where \(g_v\) denotes the centered Gaussian density on \(\mathbb R^d\) with
covariance \(vI_d\). Since \(\lambda_{ij}>0\), \(\sigma_i,\sigma_j>0\), and
\(g_v\) is strictly positive everywhere for every \(v>0\), this integral is
strictly positive. Hence \(r_{ij}(0,x;T,y)>0\) for all
\(i,j\in\mathcal I\) and all \(x,y\in\mathbb R^d\). This proves the theorem.
\end{proof}

\newpage
\bibliographystyle{plain}
\bibliography{bibliography.bib}

\end{document}